%% file: main.tex
\documentclass{amsproc}

\usepackage[T1]{fontenc}
\usepackage{microtype}
\usepackage[american]{babel}
\usepackage{pdfpages}
\usepackage{emptypage}
\usepackage{graphicx}

\usepackage{amsfonts}
\usepackage{amssymb,amsthm,amsmath,amstext,amsxtra}
\usepackage{relsize}
\usepackage{bm}       
\usepackage{bbm}       
\usepackage{mathtools} 
\usepackage{colonequals}
\usepackage{numprint} 
\usepackage{enumitem}
\usepackage{chngcntr}
\usepackage{pgfplots}
\pgfplotsset{compat=1.15}

\usepackage{csquotes} 
\usepackage[style=alphabetic, backend=biber, sorting=nyt, doi=false, isbn=false,url=false, hyperref,backref,backrefstyle=none, maxnames=20, maxalphanames=20]{biblatex}
\newbibmacro{string+doiurlisbn}[1]{%
  \iffieldundef{doi}{%
    \iffieldundef{url}{%
      \iffieldundef{isbn}{%
        \iffieldundef{issn}{%
          #1%
        }{%
          \href{https://books.google.com/books?vid=ISSN\thefield{issn}}{#1}%
        }%
      }{%
        \href{https://books.google.com/books?vid=ISBN\thefield{isbn}}{#1}%
      }%
    }{%
      \href{\thefield{url}}{#1}%
    }%
  }{%
    \href{https://dx.doi.org/\thefield{doi}}{#1}%
  }%
}
\DeclareFieldFormat{title}{\usebibmacro{string+doiurlisbn}{\mkbibemph{#1}}}
\DeclareFieldFormat[article,incollection,inproceedings,unpublished,misc,book]{title}%
    {\usebibmacro{string+doiurlisbn}{\mkbibquote{#1}}}
\DefineBibliographyStrings{english}{%
  backrefpage = {$\uparrow$},%
  backrefpages = {$\uparrow$}%
}
\usepackage{xurl}
\usepackage{breakurl}

\usepackage{xcolor}
\usepackage{color}
\usepackage[pdfencoding=auto, psdextra, breaklinks=true]{hyperref}
\definecolor{mylinkcolor}{rgb}{0.5,0.0,0.0}
\definecolor{myurlcolor}{rgb}{0.0,0.0,0.7}
\hypersetup{
 colorlinks,
 urlcolor=myurlcolor,
 citecolor=myurlcolor,
 linkcolor=mylinkcolor,
 breaklinks=true,
 pdfauthor={Raymond van Bommel, Shiva Chidambaram, Edgar Costa, Jean Kieffer}
}
\usepackage[capitalize]{cleveref}
\usepackage{tikz}
\usepackage{tikz-cd}

\counterwithin*{equation}{section}

\definecolor{mypink}{rgb}{0.85,0.08,0.48}

\newcommand{\lmfdbclass}[2]{\href{https://www.lmfdb.org/Genus2Curve/Q/#1/#2/}{#1.#2}}
\newcommand{\lmfdbcurve}[4]{\href{https://www.lmfdb.org/Genus2Curve/Q/#1/#2/#3/#4}{#1.#2.#3.#4}}

\newcommand{\Z}{\mathbb{Z}}
\newcommand{\Q}{\mathbb{Q}}
\newcommand{\R}{\mathbb{R}}
\newcommand{\C}{\mathbb{C}}
\newcommand{\F}{\mathbb{F}}
\newcommand{\Qbar}{\overline{\Q}}
\newcommand{\Fbar}{\overline{\F}}
\newcommand{\Zhat}{\widehat{\Z}}
\ifdefined\P
\renewcommand{\P}{\mathbb{P}} 
\else
\newcommand{\P}{\mathbb{P}} 
\fi
\newcommand{\Half}{\mathbb{H}} 
\newcommand{\A}{\mathcal{A}} 
\newcommand{\rhobar}{\overline{\rho}}

\newcommand{\kbar}{\overline{k}}
\newcommand{\calL}{\mathcal{L}}

\newcommand{\dual}[1]{#1^\vee} 

\DeclareMathOperator{\cond}{cond}
\DeclareMathOperator{\End}{End}
\DeclareMathOperator{\Frob}{Frob}

\DeclareMathOperator{\Gal}{Gal}
\DeclareMathOperator{\GSp}{GSp}
\DeclareMathOperator{\Jac}{Jac}

\DeclareMathOperator{\Res}{Res}
\DeclareMathOperator{\Sp}{Sp}
\DeclareMathOperator{\Diag}{Diag}

\newcommand{\Otilde}{\widetilde{O}}

\DeclareMathOperator{\im}{Im}

\DeclarePairedDelimiter{\paren}{(}{)} 
\DeclarePairedDelimiter\abs{\lvert}{\rvert}
\DeclarePairedDelimiter{\ceil}{\lceil}{\rceil}

\numberwithin{equation}{section}
\newtheorem{prop}[equation]{Proposition}
\Crefname{prop}{Proposition}{Propositions}
\newtheorem{thm}[equation]{Theorem}
\Crefname{thm}{Theorem}{Theorems}
\newtheorem{lem}[equation]{Lemma}
\newtheorem{cor}[equation]{Corollary}
\theoremstyle{definition}

\newtheorem{remark}[equation]{Remark}
\newtheorem{algorithm}[equation]{Algorithm}

\title[Computing isogeny classes of typical p.p. abelian surfaces]{Computing isogeny classes of typical principally polarized abelian surfaces over the rationals}

\author[R. van Bommel]{Raymond van Bommel}
\address{Department of Mathematics, Massachusetts Institute of Technology, Cambridge, MA 02139-4307, USA}
\email{bommel@mit.edu}
\urladdr{\url{https://raymondvanbommel.nl}}

\author[S. Chidambaram]{Shiva Chidambaram}
\address{Department of Mathematics, Massachusetts Institute of Technology, Cambridge, MA 02139-4307, USA}
\email{shivac@mit.edu}
\urladdr{\url{https://math.mit.edu/~shivac}}

\author[E. Costa]{Edgar Costa}
\address{Department of Mathematics, Massachusetts Institute of Technology, Cambridge, MA 02139-4307, USA}
\email{edgarc@mit.edu}
\urladdr{\url{https://edgarcosta.org}}

\author[J. Kieffer]{Jean Kieffer}
\address{Department of Mathematics, Harvard
  University, Cambridge, MA 02138, USA}
\email{kieffer@math.harvard.edu}
\urladdr{\url{https://scholar.harvard.edu/kieffer}}

\thanks{The first three authors were supported by Simons Foundation grant 550033.
The fourth author was supported by Simons Foundation grant 550031.}

\date{\today}

\begin{document}

\begin{abstract}
  We describe an efficient algorithm which, given a principally
  polarized~(p.p.)  abelian surface $A$ over~$\Q$ with geometric endomorphism
  ring equal to~$\Z$, computes all the other p.p.~abelian surfaces over~$\Q$
  that are isogenous to $A$.  This algorithm relies on explicit open image
  techniques for Galois representations, and we employ a combination of
  analytic and algebraic methods to efficiently prove or disprove the existence
  of isogenies.  We illustrate the practicality of our algorithm by applying it
  to \numprint{1440894} isogeny classes of Jacobians of genus 2 curves.
\end{abstract}

\maketitle

\section{Introduction}

As a consequence of Faltings' finiteness theorems for abelian
varieties~\cite{Faltings83}, the isogeny class of any abelian variety over a
number field is finite.  In the simplest case of elliptic curves over~$\Q$, we
have a good understanding of which shapes of isogeny classes can arise.
Mazur's isogeny theorem for elliptic curves~\cite{Mazur} provides the list of
primes~$\ell$ appearing as degrees of isogenies over $\Q$, namely $\ell\leq 19$
or~$\ell\in\{37,43,67,163\}$.  Furthermore, isogeny classes all have size at
most~$8$~\cite{Kenku}.  In fact, the possible isogeny graphs can be explicitly
listed~\cite[\S6]{Chiloyan}.

The standard approach to computing isogeny classes of elliptic curves over~$\Q$
is the following.  Given an elliptic curve $E$, it is enough to consider
isogenies $E\to E'$ of prime degree~$\ell$ where~$\ell$ appears in Mazur's
list.  For each such~$\ell$, computing the possible image curves~$E'$ is a
finite problem.  The computation can be carried out either by factoring
the~$\ell$-division polynomial of~$E$ and applying Vélu's formulas~\cite{Velu},
or more efficiently using explicit parametrizations or equations for the Hecke
correspondence $X_0(\ell)\to X(1)\times X(1)$~\cite{Elkies,
  CremonaWatkins}. This method was used to generate the data currently
contained in the $L$-functions and modular forms database (LMFDB)\footnote{See
  \url{https://www.lmfdb.org/EllipticCurve/Q/Source}.}~\cite{LMFDB}.

As one moves away from the case of elliptic curves over~$\Q$, very little is
known on the theoretical side.  Nevertheless, by carrying out explicit
computations, one can hope to gain insight into the possible shapes of isogeny
classes in higher dimensions.  In this paper, we make a first step in this
direction: we describe an algorithm to compute isogeny classes in the simplest
higher-dimensional case, namely that of a principally polarized~(p.p.) abelian
surface~$A$ over~$\Q$.  We also make the simplifying assumption that the
geometric endomorphism ring of~$A$ is~$\Z$, in other words, $A$ is typical,
although we plan to address more endomorphism ring types in future work.

Our overall strategy is to generalize the above method of computing isogeny
classes to the case of abelian surfaces. However, several obstacles immediately
arise:
\begin{enumerate}
\item Isogenies no longer decompose into rational, prime-degree isogenies.
\item There is no known analogue of Mazur's theorem for higher-dimensional
  abelian varieties.
\item Division polynomials are too big to be efficiently computed, except for
  very small values of~$\ell$.  The same is true of explicit equations for
  higher-dimensional analogues of the modular curve~$X_0(\ell)$ studied
  in~\cite{Milio}: see~\cite{Milio-modpols} for examples for very small
  primes~$\ell$.
\end{enumerate}

We address the first issue in \Cref{sec:isog-types}, where we show that we only
need to consider two types of isogenies of degree $\ell^2$ and $\ell^4$
respectively, where $\ell$ is a prime.

We circumvent the second issue in the following way. From any given typical
abelian surface, Serre's open image theorem for Galois
representations~\cite{SerreOeuvres} asserts that isogenies of the above types
can exist over~$\Q$ only for a finite number of primes~$\ell$ (depending on the
abelian surface). It is sufficient to consider these primes to enumerate the
``neighbors'' of~$A$ in the isogeny class. Further,
Dieulefait~\cite{Dieulefait} describes how to efficiently compute a finite
superset of this list. We review this method in \Cref{sec:subgroups}, and
provide complete proofs for the reader's convenience.

To address the third issue, we advantageously replace purely algebraic methods
by complex-analytic ones relying on the Siegel moduli space of complex
p.p.~abelian surfaces.  Concretely, given an abelian surface~$A$, we enumerate
images of its period matrix under certain Hecke correspondences, and compute
modular invariants at these points analytically.  By keeping track of the
correct scaling factors, we can actually compute these invariants as algebraic
integers (embedded in~$\C$), and thus provably recognize which of these tuples
of complex invariants correspond to abelian surfaces defined over~$\Q$.  This
algorithm is the subject of \cref{sec:analytic}.  It is significantly less
expensive than writing down equations for modular varieties or factoring
division polynomials, and is practical for isogeny degrees as large
as~$29^4 = \numprint{707281}$, the largest value we encountered in our
computations.

The output invariants only specify the $\Qbar$-isomorphism class of the abelian
surface~$A'$ isogenous to~$A$.  In \cref{sec:curves}, we explain how to obtain
the correct~$\Q$-isomorphism class using a well-known and completely algebraic
process. First, Mestre's algorithm provides a genus~2 curve over $\Q$ whose
Jacobian is isomorphic to $A'$ over $\Qbar$.  We then identify which quadratic
twist of that curve represents the desired $\Q$-isomorphism class.

The resulting algorithm has been implemented: our code and data is publicly
available at \url{https://github.com/edgarcosta/genus2isogenies}.  Numerical
computations are performed using the C library HDME~\cite{hdme}, itself based
on the Arb library~\cite{Arb} for high-precision arithmetic with
certified error bounds.

We finally discuss applications of our algorithm in \Cref{sec:examples}.  We
first present an illustrative example where we found an isogeny of
degree~$31^2$.  We also report on the results of running our algorithm on a
large dataset of Jacobians of genus 2 curves that includes the current LMFDB
data~\cite{LMFDB}.  This dataset consists of \numprint{1743737} curves split
among \numprint{1440894} isogeny classes. By completing these isogeny classes,
we find \numprint{600948} new curves.

\subsection*{Acknowledgements}
We thank Fabien Cléry for communicating us references on Siegel modular forms,
Andrew Sutherland and Noam Elkies for providing us with interesting genus~$2$
curves to use as an input to our algorithm, and Bjorn Poonen for helpful
conversations.

\section{Classification of isogenies}
\label{sec:isog-types}

In this section, we describe two fundamental isogeny types that are sufficient
to exhaust isogeny classes of typical p.p.~abelian surfaces. The essential
ingredient in this classification is that every isogeny between p.p.~abelian
varieties is compatible with the given polarizations up to the action of an
endomorphism fixed by the Rosati involution, as we detail below.

\subsection{Isogenies between p.p.~abelian varieties}

Let~$k$ be a field. For simplicity, we assume thoughout that~$k$ has characteristic
zero: this allows us to identify group schemes over~$k$ with the
groups of their $\kbar$-points endowed with an action
of~$\Gal(\kbar/k)$. Unless otherwise specified, we only consider abelian
varieties and isogenies that are defined over~$k$. Two isogenies
are considered isomorphic if they have the same domain and differ by an
isomorphism on their targets.

Let~$A$ be an abelian variety over~$k$, and denote its dual by~$A^\vee$. A
\emph{polarization} on $A$ is an isogeny $\lambda \colon A \to A^\vee$ of the
shape $a \mapsto T_a^* \mathcal{L} \otimes \mathcal{L}^{-1}$ for some (not
necessarily $k$-rational) ample line bundle $\mathcal{L}$ on $A$, where $T_a$
is the translation by $a$ in $A$.  A polarization is called \emph{principal} if
it is an isomorphism.

From now on, we assume that $A$ is principally polarized (p.p.), in other
words~$A$ is endowed with a principal polarization~$\lambda_A$. We then have a
\emph{Rosati involution} on $\End(A)$ given by
$\varphi \mapsto \lambda_A^{-1} \circ \varphi^\vee \circ \lambda_A$. An
endomorphism $\beta\in \End(A)$ is called \emph{symmetric} if it is invariant
under this involution. If~$\beta$ is symmetric, then the roots of its
characteristic polynomial are real numbers~\cite[Thm.~6 p.\,208]{Mumford}, and
we further say that $\beta$ is \emph{totally positive} if these roots are
positive.

For~$A$ as above and any integer $n$, there exists a canonical symplectic
pairing $A[n] \times A[n] \to \mu_n$ known as the \emph{Weil pairing}. More
generally, if $\beta \in \End(A)$ is any endomorphism that is symmetric and
totally positive, then $\lambda_A\circ\beta$ is another polarization of~$A$
\cite[(3) p.\,190 and (IV)
p.\,209]{Mumford}.
As in \cite[Def.~p.\,227]{Mumford}, we can also define a symplectic pairing on
$A[\beta] \times A[\beta]$ that we also refer~to as the Weil pairing. We then
have the following characterization of isogenies in terms of maximal isotropic
subgroups of torsion subgroups $A[\beta]$; see also \cite[Prop.~11.25]{EvdGM}.

\begin{lem}\label{lem:isog-subgroup}
  Let $(A, \lambda_A)$ be a p.p.~abelian variety over~$k$.  Then there is a
  one-to-one correspondence between isomorphism classes of isogenies from $A$
  to other p.p.~abelian varieties~$A'$, and pairs $(\beta, G)$, where $\beta$
  is a totally positive symmetric endomorphism of $A$, and $G$ is a subgroup of
  $A[\beta]$ which is defined over $k$ and maximal isotropic with respect to
  the Weil pairing.  Explicitly, an isogeny $\varphi \colon A \to A'$
  corresponds to the pair $(\beta, G)$ such that $\ker \varphi = G$ and $\beta$
  is the unique endomorphism satisfying
  $\varphi^\vee \circ \lambda_{A'} \circ \varphi = \lambda_A \circ \beta$.
\end{lem}

\begin{proof}
  First, we note that any isotropic subgroup of $A[\beta]$ has cardinality at
  most $\sqrt{\#A[\beta]}$, and is maximal if and only if equality holds
  \cite[Thm.~4 p.\,233]{Mumford}.

  Let $(A', \lambda_{A'})$ be a p.p.~abelian variety, and let
  $\varphi \colon A \to A'$ be an isogeny. Then there is a unique
  endomorphism~$\beta$ of~$A$ such that
  $\varphi^\vee \circ \lambda_{A'} \circ \varphi = \lambda_A \circ \beta$.
  This $\beta$ is fixed by the Rosati involution, and corresponds via the
  bijection of \cite[Application III p.\,208]{Mumford} to an ample line bundle
  on~$A$, namely the pullback by~$\varphi$ of the ample line bundle on~$A'$
  defining the polarization~$\lambda_{A'}$. Thus~$\beta$ is totally positive.
  Moreover, the pairing on $A[\beta]$ becomes trivial on $\ker{\varphi}$ by
  \cite[(1) p.\,228]{Mumford}.  As
  $\deg(\varphi)^2 = \deg(\varphi)\deg(\varphi^\vee) = \deg(\beta)$, it follows
  that $\ker{\varphi}$ is a maximal isotropic subgroup of $A[\beta]$.

  On the other hand, if $\beta$ and $G$ are given, then $\lambda_A \circ \beta$
  is a polarization on~$A$, and the
  quotient~$A/G$ is principally polarized by \cite[Cor.~p.\,231]{Mumford}.
\end{proof}

\subsection{Fundamental isogeny types} Assume now that~$k$ is a number
field. An abelian variety~$A$ over~$k$ is called {\em typical} if
$\End(A_{\overline{\Q}}) = \Z$. If~$A$ is typical, then only a small number of
isogeny types suffice in order to enumerate its isogeny class.

\begin{lem}\label{lem:isog-decomp}
  Any isogeny between typical p.p.~abelian varieties over a number field~$k$
  can be decomposed into a chain of isogenies $\varphi\colon A \to A'$ defined
  over~$k$ whose kernels are maximal isotropic subgroups of either $A[\ell]$ or
  $A[\ell^2]$, where $\ell$ is a prime number (depending on~$\varphi$).
\end{lem}
\begin{proof}
  Let $\varphi \colon A \to A'$ be any isogeny. Since~$A$ is typical, by
  \cref{lem:isog-subgroup}, there exists an integer $n\geq 1$ such that
  $K \coloneqq \ker{\varphi} \subset A[n]$ is maximal isotropic. Factor
  $n = \ell_1^{e_1} \cdot \ldots \cdot \ell_r^{e_{r}}$. Then
  $K \cap A[\ell_i^{e_i}]$ is maximal isotropic inside $A[\ell_i^{e_i}]$ for
  all $i = 1, \ldots, r$. Indeed, by \cite[Lemma~16.1]{MilneBook} the subgroup
  is isotropic, and by checking the cardinality, one sees that the subgroup is
  maximal isotropic. As each subgroup $K \cap A[\ell_i^{e_i}]$ is defined over $k$, we can
  decompose $\varphi$ as the composition of isogenies whose degrees are prime
  powers.

  Now suppose that $n = \ell^e$ is a prime power, and that $e > 2$.  Then we
  claim that
  \[
    K' \coloneqq \ell K \cap A[\ell^{e-2}] = \ell K[\ell^{e-1}]
  \]
  is a maximal isotropic subgroup of $A[\ell^{e-2}]$. We will then be able to
  decompose $\varphi$ into two isogenies of lower degree, namely $A\to A/K'$
  and $A/K' \to A/K$.

  First, we prove that $K'$ is isotropic: if
  $\langle \cdot, \cdot \rangle_{\ell^i}$ denotes the Weil pairing on
  $A[\ell^i]$ (for any~$i\geq 1$), and $\ell x_1, \ell x_2$ for
  $x_1, x_2 \in K$ are arbitrary elements of $K'$, we have that
  $\langle \ell x_1, \ell x_2 \rangle_{\ell^{e-2}} = \langle x_1, x_2
  \rangle_{\ell^e} = 1$ by \cite[Lemma~16.1]{MilneBook}. To see that~$K'$ is
  maximal, we determine its cardinality.  Let $r_1$ and $r_2$ be the ranks of
  $K[\ell]$ and~$K/K[\ell^{e-1}]$ as $\Z/\ell\Z$-modules. We have an exact
  sequence
  \[
    0 \to K[\ell] \to K[\ell^{e-1}] \stackrel{\ell}{\to} K' \to 0
  \]
  and $\#K[\ell^{e-1}] = \#K/\ell^{r_2}$. Hence
  \[
    \# K' = \# K[\ell^{e-1}]/\#K[\ell] = \#K/\ell^{r_1+r_2} = \ell^{eg - r_1 - r_2}.
  \]

  On the other hand, we have $r_1+r_2\leq 2g$ because the Weil pairing
  on~$A[\ell] \times A[\ell]$ vanishes on
  $K[\ell] \times \ell^{e-1}(K/K[\ell^{e-1}])$, and these subspaces have
  dimensions $r_1$ and $r_2$ over~$\Z/\ell\Z$ respectively. Indeed, for
  $x \in K[\ell]$ and $y \in \ell^{e-1}(K/K[\ell^{e-1}])$, we can write $x$ as
  $\ell^{e-1}x'$ for some $x' \in A[\ell^e]$ and $y$ as $\ell^{e-1}y'$ for some
  $y' \in K$, and then
  $$\langle x,y \rangle_{\ell} = \langle \ell^{e-1}x', \ell^{e-1}y'
  \rangle_{\ell} = \langle x', y'\rangle_{\ell^e}^{\ell^{e-1}} = \langle
  \ell^{e-1}x', y' \rangle_{\ell^e} = \langle x, y' \rangle_{\ell^e} = 0,$$ by
  \cite[Lemma~16.1]{MilneBook} and the fact that $K$ is isotropic.

  Since~$K'$ is isotropic for the Weil pairing on~$A[\ell^{e-2}]$ and $r_1+r_2\leq 2g$, we must have $r_1+r_2=2g$ and $K'$ is maximal
  isotropic in~$A[\ell^{e-2}]$.
\end{proof}

From now on, we focus on the case of typical p.p.~abelian surfaces. Let~$\ell$ be a prime. We say that an isogeny $\varphi:A\to A'$ is
\begin{itemize}
\item a \emph{1-step $\ell$-isogeny}, if $\ker(\varphi)$ is maximal isotropic in $A[\ell]$, and
\item a \emph{2-step $\ell$-isogeny}, if $\ker(\varphi)$ is maximal isotropic in $A[\ell^2]$ and
  isomorphic to $(\Z/\ell\Z)^2 \times \Z/\ell^2\Z$ as an abstract abelian
  group.
\end{itemize}
Note that the terminology ``2-step $\ell$-isogeny'' is slightly abusive, as the endomorphism~$\beta$ associated
with~$\varphi$ via \cref{lem:isog-subgroup} is actually $\ell^2$ in that case.

\begin{prop}\label{prop:isog-decomp}
  Any isogeny between typical p.p.~abelian surfaces over a number field~$k$ can
  be decomposed into a chain of 1-step $\ell$-isogenies, 2-step
  $\ell$-isogenies, and multiplication-by-$\ell$ endomorphisms for a series of
  primes~$\ell$, all defined over~$k$.
\end{prop}

\begin{proof}
  By \cref{lem:isog-decomp}, we only need to consider an isogeny~$\varphi$
  whose kernel~$K$ is maximal isotropic inside $A[\ell^2]$.  Then
  $\#K = \ell^4$, so $K$ is isomorphic to $(\Z/\ell\Z)^4$,
  $(\Z/\ell\Z)^2 \times \Z/\ell^2\Z$ or $(\Z/\ell^2\Z)^2$.  In the first case,
  the isogeny~$\varphi$ is multiplication by~$\ell$.  In the second case,
  $\varphi$ is 2-step.  In the third case, the group $K \cap A[\ell] = \ell K$
  is a maximal isotropic subgroup of $A[\ell]$: it is isotropic by \cite[Lemma
  16.1]{MilneBook}, and maximal for cardinality reasons. We can thus decompose
  $\varphi$ into two 1-step isogenies $A\to A/\ell K$ and $A/\ell K\to A/K$.
\end{proof}

\begin{remark}
  \label{rem:2step_isogeny_kernel}
  A 2-step isogeny factors as a composition of two 1-step
  isogenies over~$\kbar$, but these 1-step isogenies are not necessarily defined
  over~$k$.  Here is a more detailed look into this situation.  If $K$ is the
  kernel of a 2-step $\ell$-isogeny, then~$\ell K$ is a rational
  $1$-dimensional subspace of $A[\ell]$. We can extend~$\ell K$ in
  $\ell+1$ ways to get a maximal isotropic subgroup of $A[\ell]$ contained in
  $K$.  These $\ell+1$ ways to factor the 2-step isogeny as a composition of
  two 1-step isogenies are permuted by $\Gal(\kbar/k)$.
\end{remark}

\subsection{Computing an isogeny class}
\label{sub:isog-class}

From now on, we assume that $k=\Q$.  In the following sections, we
will detect the existence of 1-~or 2-step $\ell$-isogenies from~$A$ to another
p.p.~abelian surface by studying the action of $\mathrm{Gal}(\Qbar/\Q)$
on~$A[\ell]$. Let~$\calL_1$ (resp.~$\calL_2$) be the set of primes~$\ell$ such
that $A[\ell]$ admits a 1-dimensional (resp.~2-dimensional and isotropic)
Galois-stable subspace. Then the existence of a rational 1-step (resp.~2-step)
isogeny implies that~$\ell\in \calL_2$ (resp.~$\calL_1$).

Summarizing, we use the following algorithm to compute the isogeny class of a
typical p.p.~abelian surface~$A$ over $\Q$, defined as the set of isomorphism
classes of p.p.~abelian surfaces $A'$ such that there exists an isogeny
$\varphi:A\to A'$ defined over~$\Q$. These abelian surfaces all are the
Jacobian of a genus~2 curve over~$\Q$ by~\cite[Appendix, Thm.\
4]{SerreAppendix}, and this is how we encode both the input and output.

\begin{algorithm}
  \label{alg:isog-class}
  {\em Input:} a genus 2 curve $C$ over $\Q$ such that $A = \Jac(C)$.\\
  {\em Output:} the list of all p.p.~abelian surfaces over $\Q$ that are isogenous to $A$.
  \begin{enumerate}[label={\bf Step \arabic*.},labelwidth=\widthof{{\bf Step 2.}},leftmargin=!]
  \item Use Dieulefait's tests to find finite supersets of~$\calL_2$
    and~$\calL_1$ (see \S\ref{sec:subgroups}).
  \item For each~$\ell$ in these sets, compute invariants for all abelian
    surfaces~$A'$ over~$\Q$ obtained as the image of a 1-step (resp.~2-step)
    $\ell$-isogeny with domain~$A$ (see \S\ref{sec:analytic}).
  \item Reconstruct each such~$A'$ as the Jacobian of a genus~$2$ curve
    over~$\Q$ by applying Mestre's algorithm and identifying the correct twist
    (see \S\ref{sec:curves}).
  \item Repeat this process on all the newly obtained abelian surfaces as needed.
  \end{enumerate}
\end{algorithm}

\section{Rational \texorpdfstring{$\ell$}{ell}-torsion subgroups}
\label{sec:subgroups}

Let~$A$ be a typical p.p.~abelian surface over~$\Q$, and let~$N$ be its
conductor (see \cite[\S 6]{BrumerKamer} for the definition of the conductor and
further background).  Let $d$ be the maximal integer such that~$d^2 | N$.  Let
$S$ denote the set of primes of bad reduction for~$A$, i.e.~the set of primes
dividing $N$.

For each prime~$\ell \geq 2$,
let~$\rho_{\ell} \colon \Gal(\Qbar/\Q) \rightarrow \GSp(4,\Z_{\ell})$ denote
the Galois representation on the~$\ell$-adic Tate module
\[
  T_{\ell}(A) = \varprojlim\limits_{n \geq 1}A[\ell^n] \simeq \Z_\ell^4,
\]
where we fix a symplectic basis of~$T_\ell(A)$ for this last isomorphism. The
N\'eron--Ogg--Shafarevich criterion~\cite{SerreTate} states that~$\rho_{\ell}$
is unramified away from~$\ell$ and~$S$; in other words,
if~$p \notin S \cup \{\ell\}$ is a prime, then the inertia group~$I_p$ at~$p$
has trivial image under~$\rho$. The prime-to-$\ell$ part of the conductor
of~$\rho_\ell$ equals~$N$ when~$\ell$ is a prime of good reduction~\cite[Exposé
IX, \S 4]{SGA}.  For each prime $p$ of good reduction, we let
$Q_p(x) \colonequals x^4 - a_p x^3 + b_p x^2 - pa_p x + p^2 \in \Z[x]$ denote
the characteristic polynomial of~$\rho_{\ell}(\Frob_p)$, which is independent
of $\ell \neq p$. The complex roots of~$Q_p(x)$ all have absolute
value~$\sqrt{p}$.

Considering all the~$\ell$-adic representations at the same time, we obtain the
adelic Galois
representation~$\rho \colon \Gal(\Qbar/\Q) \rightarrow \GSp(4,\Zhat)$ attached
to~$A$.  Serre's open image theorem~\cite{SerreOeuvres} asserts that the image
of~$\rho$ is an open subgroup of~$\GSp(4,\smash{\Zhat})$, or equivalently has
finite index in it.  Consequently, the mod-$\ell$ Galois
representation~$\rhobar_{\ell} \coloneqq \rho_{\ell} \bmod{\ell}$, attached to
the Galois action on~$A[\ell]$, is surjective for all primes~$\ell$ outside a
finite set.

Dieulefait~\cite{Dieulefait} describes an algorithm to explicitly determine a
finite set of primes containing all primes~$\ell$ at which $\rhobar_{\ell}$ is
not surjective, using the classification of maximal subgroups
of~$\GSp(4,\F_{\ell})$.  For each type of maximal subgroup, one can give
necessary conditions on $\ell$ for the image of~$\rhobar_{\ell}$ to be
contained in a subgroup of this type, and these conditions are satisfied by
finitely many primes~$\ell$.  The algorithm has been implemented in
SageMath~\cite{Sage} by the work of~\cite{galreps}.

In this work, we are only interested in two types of maximal subgroups, namely
the stabilizers of lines and two-dimensional isotropic subspaces in
~$\F_\ell^4$~\cite[Lemma 2.3(1)]{galreps}. Keeping notation from~\S
\ref{sub:isog-class}, we call~$\calL_1$ (resp.~$\calL_2$) the finite set of
primes~$\ell$ for which the image of~$\rhobar_\ell$ stabilizes a line (resp.~a
2-dimensional isotropic subspace).  A 1-step $\ell$-isogeny with domain~$A$
exists over~$\Q$ if and only if $\ell\in \calL_2$.  Moreover, if a 2-step
$\ell$-isogeny with domain~$A$ exists over~$\Q$, then~$\ell\in \calL_1$ as per
\cref{rem:2step_isogeny_kernel}.  We now describe the associated Dieulefait
criteria~\cite[\S3.1,\,\S3.2]{Dieulefait} in more detail, assuming from now on
that~$A$ has good reduction at~$\ell$.

\subsection{Computing a finite superset of \texorpdfstring{$\calL_1$}{L1}}
\label{sec:computing L1}

Dieulefait's tests rely on the factorization of characteristic
polynomials~$Q_p(x)$ over finite fields, for a given list of primes~$p$. We
include this list as part of the input of the following algorithm.

\begin{algorithm}\label{alg:L1}
  {\em Input:}
  \begin{itemize}
  \item a genus 2 curve $C$ over $\Q$ such that $A = \Jac(C)$,
  \item the conductor $N$ of $A$,
  \item and a non-empty finite set $P$ of primes of good reduction for $C$.
  \end{itemize}
  {\em Output:} a finite superset of $\calL_1$.
  \begin{enumerate}[label={\bf Step \arabic*.},labelwidth=\widthof{{\bf Step 2.}},leftmargin=!]
  \item Compute~the maximal integer $d$ such that $d^2 \mid N$.
  \item Compute~$Q_p\in \Z[x]$ for $p \in P$
    by computing the number of points on~$C$ over~$\F_p$ and~$\F_{p^2}$.
  \item Compute~$M = \gcd_{p\in P} \Res\paren[\big]{Q_p(x), x^{f(p)}-1}$, where
    $f(p)$ is the order of~$p$ in~$(\Z/d\Z)^{\times}$.
  \item Return the list of prime divisors of~$M$.
  \end{enumerate}
\end{algorithm}

\begin{prop}
\cref{alg:L1} returns a finite superset of $\calL_1$.
\end{prop}
\begin{proof}
  Suppose that there exists a~$1$-dimensional subrepresentation~$(\pi,V)$
  of~$\rhobar_{\ell}$, induced by a stable line~$V\subset A[\ell]$.
  Let~$V^{\perp}$ denote the subgroup of~$A[\ell]$ that pairs trivially
  with~$V$ under the Weil pairing. By Galois equivariance of the Weil pairing,
  the 1-dimensional quotient representation on~$A[\ell]/V^{\perp}$ is given
  by~$\pi^{-1}\chi_{\ell}$, where~$\chi_{\ell}$ is the mod-$\ell$ cyclotomic
  character. By results of Raynaud~\cite[Cor.\ 3.4.4]{Raynaud} and
  Serre~\cite[\S 1.9]{Serre72}, we know that~$\pi$ restricted to the inertia
  group~$I_{\ell}$ at~$\ell$ is either trivial or equal to~$\chi_{\ell}$.

  In any case, there exists a character~$\varepsilon$, unramified away
  from~$N$, such that the semisimplification~$\rhobar_{\ell}^{ss}$ admits
  both~$\varepsilon$ and~$\varepsilon^{-1}\chi_{\ell}$ as direct summands.
  Hence the conductor of~$\rhobar_{\ell}$ is divisible by the square of the
  conductor of~$\varepsilon$, in other words~$\cond(\varepsilon)$ divides~$d$.
  Class field theory then implies that~$\varepsilon$ is a character
  of~$\Gal(\Q(\zeta_d)/\Q) \simeq (\Z/d\Z)^{\times}$. For each prime~$p \in L$,
  since $f(p)$ is the order of~$p$ in~$(\Z/d\Z)^{\times}$, we have
  $\varepsilon(\Frob_p)^{f(p)} = 1$. Therefore,~$\varepsilon(\Frob_p)$ is a
  root of the characteristic polynomial of~$\rhobar_{\ell}(\Frob_p)$ and also
  of the polynomial~$x^{f(p)}-1$ over~$\F_{\ell}$. Thus $\ell$ divides
  $\Res(Q_p(x), x^{f(p)}-1)$ and hence $\ell$ divides $M$.

  Since the complex roots of $Q_p(x)$ have absolute value $\sqrt{p}$, they are
  distinct from the roots of $x^{f(p)} - 1$ which have absolute value $1$. So
  the resultants and thus $M$ are guaranteed to be non-zero and the computed
  superset is finite.
\end{proof}

\begin{remark}
  \label{rem:120}
  Building on results of Grothendieck~\cite[Exposé IX, Prop.\ 3.5]{SGA} and
  Larson--Vaintrob~\cite[Thm.\ 7.2]{LarsonVaintrob}, the authors
  of~\cite{galreps} introduce the following strengthening of the above
  technique (Alg.~3.3). For~$p\in P$, let~$r =
  \gcd(f(p),120)$. Let~$R_p$ be the polynomial whose roots are the $r$-th
  powers of the roots of~$Q_p$. Then each~$\ell\in \calL_1$ must
  divide~$p\, R_p(1)$. The original criterion by Dieulefait, as presented
  above, was however sufficient for our purposes.
\end{remark}

\subsection{Computing a finite superset of \texorpdfstring{$\calL_2$}{L2}}
\label{sec:computing L2}

In the case of~$\calL_2$, the computation of a finite superset might fail if
the list of auxiliary primes~$p$ is too small, leading to the following
algorithm.

\begin{algorithm}
  \label{alg:L2}
  {\em Input:}
  \begin{itemize}
  \item a genus 2 curve $C$ over $\Q$ such that $A = \Jac(C)$,
  \item the conductor $N$ of $A$,
  \item and a non-empty finite set $P$ of primes of good reduction for $C$.
  \end{itemize}
  {\em Output:} a finite superset of $\calL_2$, or \texttt{false}.
  \begin{enumerate}[label={\bf Step \arabic*.},labelwidth=\widthof{{\bf Step 2.}},leftmargin=!]
  \item Compute~$d$.
  \item For $p \in P$, compute $Q_p = x^4 - a_p x^3 + b_p x^2 - pa_p x + p^2 \in \Z[x]$.
  \item For $p \in P$, compute the polynomials
  \begin{align*}
  R_{1,p}(x) &\coloneqq (b_p x -1-p^2x^2)(px+1)^2 - a_p^2 p x^2 \quad \text{and}\\
  R_{2,p}(x) &\coloneqq (b_p x -p -px^2)(x+1)^2-a_p^2x^2.
  \end{align*}
\item For~$i=1$ and~$2$, compute
  $M_i = \gcd_{p\in P} \Res\paren[\big]{R_{i,p}(x), x^{f(p)}-1}$, where $f(p)$
  is the order of~$p$ in~$(\Z/d\Z)^{\times}$. Let~$M=M_1M_2$.
\item If $M$ is nonzero, return the list of prime divisors of $M$, else
  return \texttt{false}.
  \end{enumerate}
\end{algorithm}

\begin{prop}
    \cref{alg:L2} returns either\, {\tt false} or a finite superset of $\calL_2$.
\end{prop}
\begin{proof}
  Suppose that~$\rhobar_{\ell}$ admits a $2$-dimensional isotropic
  subrepresentation denoted by~$(\pi,V)$.  The quotient representation on~$A[\ell]/V$ is
  given by~$\chi_{\ell} \otimes (\pi^{-1})^t$. By~\cite[Cor.\ 3.4.4]{Raynaud}
  and~\cite[\S 1.9]{Serre72}, we obtain as above that
  either~$\det(\pi) = \varepsilon\chi_{\ell}^2$
  or~$\det(\pi) = \varepsilon\chi_{\ell}$ for some character~$\varepsilon$
  unramified away from~$N$.  The direct sum
  decomposition~$\rhobar_{\ell}^{ss} \simeq \pi \oplus \chi_{\ell} \otimes
  (\pi^{-1})^t$ then implies that~$\cond(\varepsilon)$ still divides~$d$.

  If $\det(\pi) = \varepsilon\chi_{\ell}^2$, then for every prime $p$ of good
  reduction, we have the following factorization of~$Q_p(x)$ into a product of
  two related quadratic polynomials modulo~$\ell$:
  \[
    Q_p(x) = \paren[\big]{x^2-rx+p^2\varepsilon(p)}
    \paren[\Big]{x^2-\frac{r}{p\varepsilon(p)}x+\varepsilon^{-1}(p)}.
  \]
  If $\det(\pi) = \varepsilon\chi_{\ell}$, then for all such $p$, we have
  \[
    Q_p(x) = \paren[\big]{x^2-rx+p\varepsilon(p)}
    \paren[\Big]{x^2-\frac{r}{\varepsilon(p)}x+p\varepsilon^{-1}(p)}.
  \]

  By comparing coefficients and eliminating~$r$, one observes that the first
  kind of factorization happens if and only if $\varepsilon(p)$ is a root of
  the integral polynomial~$R_{1,p}(x)$ introduced in \Cref{alg:L2}
  over~$\F_\ell$\ for all $p$.  Similarly, the second kind of factorization
  happens if and only if $\varepsilon(p)$ is a root of~$R_{2,p}(x)$
  over~$\F_{\ell}$\ for all $p$.  Thus, there is an $i\in\{1,2\}$ such that for
  all $p$ of good reduction,
  $\Res\paren[\big]{R_{i,p}(x),x^{f(p)}-1} = 0 \mod{\ell}$. We deduce that
  $\ell$ always divides $M=M_1M_2$.
\end{proof}

We now show that \cref{alg:L2} returns a superset of $\calL_2$ provided
that~$P$ contains enough primes.  In practice, failures are not an issue even
for a small list~$P$.
\begin{prop} \label{prop:finiteset_reducible2} For $B$ large enough and
  $P = \{p \leq B \colon p \textup{ is a good prime for }C\}$, \cref{alg:L2}
  returns a finite list of primes, in other words both~$M_1$ and~$M_2$ are
  nonzero.
\end{prop}
\begin{proof} We first prove that $M_1\neq 0$.  For~$p = 1 \bmod{d}$, we have
  \[
  \Res(R_{1,p}(x), x^{f(p)}-1) = R_{1,p}(1) = (b_p-1-p^2)(p+1)^2-a_p^2p.
  \]
  Using the bounds~$\abs{a_p} \leq 4\sqrt{p}$ and $\abs{b_p} \leq 6p$ coming
  from the fact that the roots of~$Q_p(x)$ have absolute value~$\sqrt{p}$, it
  follows that $R_{1,p}(1) < 0$ if~$p$ is large enough, and thus $M_1 \neq 0$
  for large~$B$.

  Second, we prove that~$M_2 \neq 0$. Consider all polynomials over~$\F_\ell$
  which, over~$\Fbar_\ell$, admit a factorization of the form
  $$
  (x^2-rx+p\eta) \paren{x^2-\tfrac{r}{\eta}x + p\eta^{-1}},
  $$
  where~$\eta\in\Fbar_\ell$ is a $d$-th root of unity and~$r\in
  \Fbar_\ell$. When~$\ell$ is sufficiently large, these polynomials do not
  account for all characteristic polynomials of matrices
  in~$\GSp(4,\F_\ell)$. By Serre's open image theorem, we can further assume
  that the representation~$\rhobar_\ell$ is surjective. For such an~$\ell$, by
  the Chebotarev density theorem, there must exist infinitely many~$p$ such
  that~$Q_p(x)$ does not factor over~$\Fbar_\ell$ in the above shape, and thus
  $\ell\nmid\Res(R_{2,p}(x), x^{f(p)}-1)$.  Therefore $M_2 \neq 0$ if~$B$ is
  sufficiently large.
\end{proof}

\begin{remark}
  Showing that~$M_2\neq 0$ for a sufficiently large~$B$ is non-trivial: indeed,
  if~the reduction of $A$ modulo $p$ is isogenous to the square of an elliptic
  curve over~$\F_p$, then $\Res(R_{2,p}(x), x^{f(p)}-1)=0$.
\end{remark}

\begin{remark}
  An improvement of a similar flavor to \Cref{rem:120} is also available here:
  see \cite[\S 3.1.3]{galreps}.
\end{remark}

\subsection{Other tests for irreducibility} The output of \Cref{alg:L1,alg:L2}
usually consists of very short lists of primes, but might still contain
extraneous primes~$\ell\notin \calL_1\cup\calL_2$. In order to further weed out
some of these primes, one can compute~$Q_p(x)$ for a larger set of primes~$p$,
and eliminate any prime~$\ell$ with the property that one of these polynomials
is irreducible modulo~$\ell$.  In our computations, we considered all
primes~$p\leq 500$ of good reduction for~$C$.

\section{Invariants of isogenous abelian surfaces}
\label{sec:analytic}

In this section, we describe an efficient algorithm solving the following
problem: given a typical p.p.~abelian surface $A$ over~$\Q$ and a prime
number~$\ell$, compute the complete list of p.p.~abelian surfaces~$A'$
over~$\Q$ such that~$A$ and~$A'$ are linked by a rational 1-~or 2-step
$\ell$-isogeny.

Devising a polynomial-time algorithm for this task is
straightforward: we can write down equations for the torsion
subgroups~$A[\ell]$ or~$A[\ell^2]$, look for rational subgroups of the correct
shape by factoring these polynomials over~$\Q$, and apply algorithms to
compute quotients of p.p.~abelian surfaces by isotropic
subgroups~\cite{CouveignesEzome, LubiczRobert}. Such an algorithm would however
be hopelessly slow in practice.

A more efficient approach is to use modular equations for p.p.~abelian
surfaces, which are higher-dimensional analogues of elliptic modular polynomials:
see~\cite{BroekerLauter} for their definition in the case of 1-step isogenies,
and~\cite{Milio} for an efficient algorithm to compute them.  Evaluating
modular equations at~$A$ provides tuples of modular invariants (for instance
Igusa--Clebsch invariants) of abelian surfaces isogenous to~$A$, possibly
defined over a number field. This evaluation can be done within a reasonable
complexity, namely~$\Otilde(\ell^6h)$ bit operations\footnote{We use the
  notation $\Otilde(N)$ to denote~$O\paren[\big]{N\log^k(N)}$ for some value
  of~$k$.} in the case of 1-step $\ell$-isogenies, where~$h$ is the height of
the invariants of~$A$~\cite{KiefferEvaluation}. This algorithm works by
computing the invariants of isogenous abelian surfaces as complex numbers,
packaging them into a polynomial, and recognizing its coefficients as rational
numbers. Then a rational isogeny from~$A$ exists exactly when this polynomial
has a rational root (see \cref{prop:ratl-invariants} below).

Here we take this method one step further: we compute these invariants
over~$\C$, and directly recognize when they are attached to a p.p.~abelian
surface defined over~$\Q$. In a sense, we detect rational roots of modular
equations without computing the number fields that other roots generate. When
doing so, the cost is further lowered to~$\Otilde\paren[\big]{(n+1)\ell^d h}$
bit operations, where~$d=3$ (resp.~$4$) in the case of 1-step (resp.~2-step)
isogenies, and~$n$ is the number of roots ``close to'' being rational, in a
precise sense explained below. (We usually have~$n=0$, sometimes $n=1$, and
rarely more.) Crucially, this analytic method allows for certification. The
modular invariants we compute provably correspond to p.p.~abelian surfaces
rationally isogenous to~$A$, and provably miss none of them.

The computations over~$\C$ make use of structure theorems for Siegel modular
forms in dimension~$2$ and explicit formulas for Hecke correspondences,
recalled in~\S\ref{sub:siegel-mf} and~\S\ref{sub:hecke}.  Next, we explain why
detecting rational isogenies reduces to detecting invariants defined
over~$\Q$~(\S\ref{sub:ratl-invariants}), and how to compute these invariants as
algebraic integers, which is the crucial idea behind
certification~(\S\ref{sub:certification}). Finally, in~\S\ref{sub:outline}
and~\S\ref{sub:implementation}, we describe the algorithm and sketch its
complexity analysis.

\subsection{Siegel modular forms}
\label{sub:siegel-mf}

Over~$\C$, every p.p.~abelian surface~$A$ can be written as a complex
torus~$\C^2/(\Z^2\oplus\tau\Z^2)$, where~$\tau$ belongs to the Siegel upper
half space~$\Half_2$, consisting of complex~$2\times 2$ symmetric matrices with
positive definite imaginary part. Such a~$\tau$ is called a (small)
\emph{period matrix} of~$A$. The group~$\GSp(4,\R)^+$ consisting of general
symplectic matrices with positive similitude factor acts on~$\Half_2$ as
follows:
\begin{displaymath}
  \gamma\tau = (a\tau+b)(c\tau+d)^{-1}, \quad \text{where }
  \gamma = \left(\begin{matrix} a&b\\c&d \end{matrix}\right)
  \text{ in $2\times 2$ blocks.}
\end{displaymath}
For later use, we also write
\begin{displaymath}
  \gamma^*\tau = c\tau + d.
\end{displaymath}
The period matrix of~$A$ is unique up to the action of the modular
group~$\Sp(4,\Z)$, so the quotient~$\Sp(4,\Z)\backslash\Half_2 $ is precisely
the coarse moduli space of complex p.p.~abelian surfaces.  In fact, this coarse
moduli space~$\A_2$ exists as a quasi-projective
variety defined over $\Q$, and $\A_2(\C)$ can be identified with
$\Sp(4,\Z)\backslash\Half_2 $.

A (scalar-valued) \emph{Siegel modular form} on~$\Half_2$ of
weight~$k\in \Z_{\geq 0}$ (and level 1) is a complex-analytic map
$f \colon \Half_2\to \C$ satisfying
$f(\gamma\tau) = \det(\gamma^*\tau)^k f(\tau)$ for all~$\tau\in \Half_2$
and~$\gamma\in\Sp(4,\Z)$. See~\cite{vanderGeer} for more background on
these objects. Siegel modular forms admit Fourier
expansions: writing~$\tau\in \Half_2$ as
\begin{displaymath}
  \tau = \left(
  \begin{matrix}
    \tau_1&\tau_3\\\tau_3&\tau_2
  \end{matrix}\right)
\end{displaymath}
and~$q_j = \exp(2\pi i \tau_j)$ for~$1\leq j\leq 3$, the Fourier expansion of a
Siegel modular form belongs to the power series
ring~$\C[q_3,q_3^{-1}][[q_1,q_2]]$.

By the Baily--Borel theorem, Siegel modular forms with rational Fourier
coefficients yield projective embeddings of~$\A_2$ that are defined over~$\Q$,
and can thus be used as rational coordinates on this moduli
space. Igusa~\cite{Igusa1} proved the following fundamental theorem.

\begin{thm}
  \label{thm:mf-gens}
  The graded~$\C$-algebra of Siegel modular forms on~$\Half_2$ of even weight
  is free with four generators~$M_4,M_6,M_{10},M_{12}$ of weights~$4,6,10,12$
  with integral Fourier coefficients.
\end{thm}

In this paper, we normalize these generators so that their Fourier expansions
are primitive and~$M_{10},M_{12}$ are cusp forms. This defines them uniquely up
to sign. We can fix these signs by specifying their first few Fourier
coefficients:
\begin{displaymath}
  \begin{aligned}
    M_4(\tau) &= 1 + 240(q_1 + q_2) +
    O\bigl(q_1^2,q_2^2,q_1q_2 \bigr), \\
    M_6(\tau) &= 1 - 504(q_1 + q_2) +
    O\bigl(q_1^2,q_2^2,q_1q_2 \bigr), \\
    M_{10}(\tau) &= \bigl(q_3 - 2 + q_3^{-1}\bigr) q_1 q_2 + O(q_1^2,q_2^2),
    \qquad\text{and} \\
    M_{12}(\tau) &= \bigl(q_3 + 10 + q_3^{-1}\bigr) q_1 q_2 + O\bigl(q_1^2,q_2^2 \bigr).
    \\
  \end{aligned}
\end{displaymath}
We find this normalization more convenient than the ones usually considered in the literature.
In terms of Igusa's notation in~\cite{Igusa1}, we have
\begin{displaymath}
  M_4 = \psi_4,\quad M_6 = \psi_6,\quad M_{10} = 4\chi_{10}, \quad \text{and}
  \quad M_{12} = 12\chi_{12}.
\end{displaymath}
In terms of the modular forms~$h_k$ for~$k\in \{4,6,10,12\}$ from~\cite[\S7.1]{Streng}, we have
\begin{equation}
\label{eqn:m-to-h}
  M_4 = 2^{-2} h_4, \quad M_6 = 2^{-2} h_6, \quad M_{10} = -2^{-12} h_{10},\quad
  \text{and} \quad M_{12} = 2^{-15} h_{12}.
\end{equation}
From \cref{thm:mf-gens}, we deduce that for each~$\tau\in \Half_2$, at least one of the
values $M_k(\tau)$ for $k\in \{4,6,10,12\}$ does not vanish.

Igusa~\cite{Igusa2} further determined an explicit set of fourteen generators
for the graded ring of Siegel modular forms with integral Fourier coefficients,
which contains the above forms~$M_k$.  The following easy corollary of Igusa's
result will play an essential role in this paper.

\begin{prop}[{\cite[\S2.1]{KiefferEvaluation}}]
  \label{prop:int-gens}
  Let~$f$ be a Siegel modular form on~$\Half_2$ of even weight~$k$ with
  integral Fourier coefficients. Then~$12^kf\in \Z[M_4,M_6,M_{10}, M_{12}]$.
\end{prop}

If~$f$ is a Siegel modular form of weight~$k$ with rational Fourier
coefficients, then we can also give~$f$ an algebraic meaning as
follows~\cite[Def.~1.1 p.\,137]{FaltingsChai}. Let~$A$ be a p.p.~abelian
surface over a number field~$L$ embedded in~$\C$, and let~$\omega$ be a basis
of the $1$-dimensional $L$-vector space $\wedge^2 \Omega^1(A)$. Then
$f(A,\omega)$ is a well-defined element of~$L$, and satisfies
$f(A, t\omega) = t^{-k} f(A,\omega)$ for all $t\in L^\times$.

The relation between $f(A,\omega)$ and the values of~$f$ on the Siegel half space~$\Half_2$ is the
following \cite[p.\,141]{FaltingsChai}. Let~$\tau\in \Half_2$ be a period
matrix of~$A$, and choose an
isomorphism~$\eta \colon A(\C)\to \C^2/(\Z^2\oplus\tau\Z^2)$. The basis of
differentials $(2\pi i\, dz_1, 2\pi i\, dz_2)$ induces a natural basis
of~$\wedge^2\Omega^1$ on the complex torus; call this
basis~$\omega(\tau)$. There is a unique~$r\in \C^\times$ such that
$\omega = r\cdot \eta^*\omega(\tau)$. Then $f(A,\omega) = r^{-k} f(\tau)$; one can check that this
quantity~$f(A,\omega)$ does not depend on the choice of~$\tau$ or~$\eta$.

Let now~$A$ be a p.p.~abelian surface over~$\Q$, and choose a basis $\omega$ of
$\wedge^2\Omega^1(A)$. Then the weighted projective point
\begin{equation}\label{eqn:def-mk}
  \paren[\big]{M_4(A,\omega) : M_6(A,\omega) : M_{10}(A,\omega) : M_{12}(A,\omega)}
  \in \P^{4,6,10,12}(\Q)
\end{equation}
is independent of~$\omega$, since scaling $\omega$ by $t^{-1}\in \Q^\times$ scales
the above coordinates by~$(t^4,t^6,t^{10},t^{12})$. We call this projective
point the \emph{modular invariants} of~$A$. This projective point has a unique
representative $(m_4,m_6,m_{10},m_{12}) \in \Z^4$ that is reduced in the sense
that no prime $p$ satisfies $p^k \mid m_k$ for all $k\in\{4,6,10,12\}$; by a
slight abuse of language, we also call this tuple of integers the modular
invariants of~$A$.

If~$A$ is the Jacobian of a genus~$2$ curve~$C$ defined over~$\Q$, then the
modular invariants of~$A$ can be computed as follows.
Let~$(I_2:I_4:I_6:I_{10}) \in \P^{2,4,6,10}(\Q)$ be the Igusa--Clebsch
invariants of~$C$ as defined in~\cite[\S2.1]{Streng}; these invariants are also
denoted by $(A:B:C:D)$ in~\cite{Igusa1} and $(A':B':C':D')$
in~\cite{Mestre}. Then, as a consequence of \eqref{eqn:m-to-h}, the modular
invariants of~$A$ are
\begin{equation}\label{eqn:m-to-I}
  (m_4:m_6:m_{10}:m_{12}) = \paren[\big]{2^{-2} I_4 : 2^{-3} \paren{ I_2 I_4 - 3 I_6}
    : - 2^{-12} I_{10} : 2^{-15} I_2 I_{10}}.
\end{equation}
In particular, on the input of~$C$, the modular invariants of~$A$ can easily be
computed from the expression of $I_2,\ldots, I_{10}$ as polynomials in the
coefficients of~$C$.

\subsection{Hecke correspondences}
\label{sub:hecke}

Consider a period matrix~$\tau\in \Half_2$ attached to a p.p. abelian
surface~$A$ over~$\C$. Then the period matrices of abelian surfaces linked
to~$A$ by an isogeny of a given type can be computed by letting certain
symplectic matrices act on~$\tau$. This precisely corresponds to analytic
formulas for the action of Hecke operators on spaces of Siegel modular
forms~\cite[\S10]{CleryvanderGeer},~\cite[Chap.\ VI, \S5]{Krieg}.
The case of 1-step isogenies corresponds to the Hecke operator usually denoted
by~$T(\ell)$, and 2-step isogenies correspond to the Hecke
operator~$T_{1}(\ell^2)$.  Concretely, we define the following collections of
matrices:
\begin{enumerate}
\item $S(\ell)$ consists of the~$\ell^3+\ell^2+\ell+1$ matrices of the form
  \begin{displaymath}
    \left(
      \begin{smallmatrix}
        1&0&a&b\\
        0&1&b&c\\
        0&0&\ell&0\\
        0&0&0&\ell
      \end{smallmatrix}
    \right),\quad
    \left(
      \begin{smallmatrix}
        \ell&0&0&0\\
        -a&1&0&b\\
        0&0&1&a\\
        0&0&0&\ell
      \end{smallmatrix}
    \right),\quad
    \left(
      \begin{smallmatrix}
        1&0&a&0\\
        0&\ell&0&0\\
        0&0&\ell&0\\
        0&0&0&1
      \end{smallmatrix}
    \right)\quad \text{or}\quad
    \left(
      \begin{smallmatrix}
        \ell&0&0&0\\
        0&\ell&0&0\\
        0&0&1&0\\
        0&0&0&1
      \end{smallmatrix}
    \right)
  \end{displaymath}
  where $a,b,c$ run through $\{0,\ldots,\ell-1\}$;
\item $S(\ell^2)$ consists of the~$\ell^4+\ell^3+\ell^2+\ell$ matrices of the
  following form:
  \begin{displaymath}
    \left(
      \begin{smallmatrix}
        \ell&0&0&a\ell\\
        -b&1&a&ab+d\\
        0&0&\ell&b\ell\\
        0&0&0&\ell^2
      \end{smallmatrix}
    \right),\quad
    \left(
      \begin{smallmatrix}
        1&0&d&-a\\
        0&\ell&-\ell a&0\\
        0&0&\ell^2&0\\
        0&0&0&\ell
      \end{smallmatrix}
    \right),\quad
    \left(
      \begin{smallmatrix}
        \ell^2&0&0&0\\
        -a\ell&\ell&0&0\\
        0&0&1&a\\
        0&0&0&\ell
      \end{smallmatrix}
    \right)\quad \text{or}\quad
    \left(
      \begin{smallmatrix}
        \ell&0&0&0\\
        0&\ell^2&0&0\\
        0&0&\ell&0\\
        0&0&0&1
      \end{smallmatrix}
    \right)
  \end{displaymath}
  where $a,b$ run through~$\{0,\ldots,\ell-1\}$ and~$d$ runs
  through~$\{0,\ldots,\ell^2-1\}$; as well as
  \begin{displaymath}
    \left(
      \begin{smallmatrix}
        \ell&0&a&b\\
        0&\ell&b&c\\
        0&0&\ell&0\\
        0&0&0&\ell
      \end{smallmatrix}
    \right)
  \end{displaymath}
  where $a,b,c$ run through~$\{0,\ldots,\ell-1\}$ with the additional
  conditions that~$ac=b^2$ and $(a,b,c)\neq (0,0,0)$.
\end{enumerate}
The set~$S(\ell)$ (resp.~$S(\ell^2)$) consists of matrices in~$\GSp(4,\R)^+$
with integral coefficients, zero lower left block, and with the property that the determinant of their lower right block divides~$\ell^2$ (resp.~$\ell^3$).

\begin{prop}
  \label{prop:hecke}
  Let~$\tau\in \Half_2$, let~$A = \C^2/(\Z^2\oplus\tau\Z^2)$ be the
  p.p.~complex abelian surface attached to~$\tau$, and let~$\ell$ be a prime
  number.
  \begin{enumerate}
  \item \label{it:hecke-1} The matrices~$\gamma\tau$ for~$\gamma\in S(\ell)$ are period matrices
    for the abelian surfaces~$A/K$ where~$K$ runs through the maximal
    isotropic subgroups of~$A[\ell]$.
  \item \label{it:hecke-2} The matrices~$\gamma\tau$ for~$\gamma\in S(\ell^2)$
    are period matrices for the abelian surfaces~$A/K$ where~$K$ runs through the
    maximal isotropic subgroups of~$A[\ell^2]$ isomorphic
    to~$(\Z/\ell\Z)^2\times (\Z/\ell^2\Z)$.
  \end{enumerate}
\end{prop}

\begin{proof}
  First, we consider $1$-step isogenies. Consider the subgroup $\Gamma^0(\ell)$
  of $\Sp_4(\Z)$ defined as
  \[
    \Gamma^0(\ell) \coloneqq \left\{\begin{pmatrix}a&b\\c&d\end{pmatrix}\in
      \Sp(4,\Z): b=0\pmod\ell\right\}.
  \]
  The map
  \[
    \tau \mapsto \paren[\big]{\C^2/(\Z^2\oplus\tau\Z^2), (\Z^2\oplus
      \tfrac1\ell\tau\Z^2)/(\Z^2\oplus\tau\Z^2)}
  \]
  is a bijection between $\Gamma^0(\ell)\backslash\Half_2$ and the set of
  isomorphism classes of pairs~$(A,K)$, where~$A$ is a p.p.~abelian surface
  over~$\C$ and~$K\subset A[\ell]$ is maximal isotropic for the Weil
  pairing. (This is proved in an analogous way to \cite[Thm.\
  3.2]{BroekerLauter}, which uses the subgroup $\Gamma_0(\ell)$ instead.) The
  quotient surface~$A/K$ admits~$\frac1\ell \tau$ as a period matrix. The
  period matrices we wish to enumerate are therefore the
  $\frac1\ell\gamma\tau$, where $\gamma$ runs through a (finite) set of
  representatives of $\Gamma^0(\ell)\backslash\Sp(4,\Z).$

  We can rewrite these matrices using the action of~$\GSp(4,\Q)^+$, noting that
  for all $\gamma\in \Sp(4,\Z)$,
  \[
    \paren[\big]{\Diag(1,1,\ell,\ell)\gamma}\cdot \tau = \tfrac1\ell(\gamma\tau).
  \]
  This leads us to the formalism of Hecke operators in terms of double
  cosets: we have
\[
    \Gamma^0(\ell) = \Sp(4,\Z) \cap \Diag(1,1,\ell,\ell)^{-1} \Sp(4,\Z) \Diag(1,1,\ell,\ell),
\]
so the map $\gamma\mapsto \Diag(1,1,\ell,\ell)\gamma$ induces a bijection
\[
    \Gamma^0(\ell) \backslash \Sp(4,\Z) \to \Sp(4,\Z) \backslash \Sp(4,\Z) \Diag(1,1,\ell,\ell)\Sp(4,\Z).
\]
By~\cite[Prop.\ 10.1]{CleryvanderGeer}, the set $S(\ell)$ is a set of representatives for the coset space on the right hand side (we acted on some representatives by elements of~$\Sp_4(\Z)$ to simplify them). This proves~\eqref{it:hecke-1}.

The $2$-step case is similar. The reduction map
$\Sp(4,\Z)\to \Sp(4,\Z/\ell^2\Z)$ is surjective (see \cite[\S
3]{BroekerLauter}), and we define $\Gamma^0(\ell^2)$ as the preimage in
$\Sp(4,\Z)$ of the stabilizer of the maximal isotropic subgroup
$\bigl\langle (0,\ell,0,0), (0,0,1,0), (0,0,0,\ell) \bigr\rangle$
in~$(\Z/\ell^2\Z)^4$. The map
\[
  \tau\mapsto \paren[\big]{\C^2/(\Z^2\oplus\tau \Z^2), (\Z\oplus
    \tfrac1\ell\Z)\oplus \tau(\tfrac{1}{\ell^2}\Z \oplus
    \tfrac1\ell\Z))/(\Z^2\oplus\tau\Z^2)}
\]
is a bijection between $\Gamma^0(\ell^2)\backslash \Half_2$ and the set
of isomorphism classes of pairs~$(A,K)$, where $A$ is a p.p.~abelian surface
over~$\C$ and $K\subset A[\ell^2]$ is a maximal isotropic subgroup for the Weil
pairing and isomorphic to $(\Z/\ell\Z)^2\times \Z/\ell^2\Z$. A period matrix
for the quotient abelian surface $A/K$ is then
$\Diag(1,\ell,\ell^2,\ell)\,\tau$. As above, a reformulation in terms of double
cosets will be convenient. We have
\[
  \Gamma^0(\ell^2) = \Sp(4,\Z) \cap \Diag(1,\ell,\ell^2,\ell)^{-1}
  \Sp(4,\Z)\Diag(1,\ell,\ell^2,\ell),
\]
so the map $\gamma \mapsto \Diag(1,\ell,\ell^2,\ell)\gamma$ induces a bijection
\[
  \Gamma^0(\ell^2)\backslash \Sp(4,\Z) \to \Sp(4,\Z) \backslash \Sp(4,\Z)
  \Diag(1,\ell,\ell^2,\ell) \Sp(4,\Z).
\]
Combining~\cite[Prop.\ 10.5]{CleryvanderGeer} with \cite[Chap.\ VI,
Lem.~5.2]{Krieg}, we find that $S(\ell^2)$ is exactly a set of representatives
for the coset space on the right hand side.
\end{proof}

\subsection{Rational invariants versus rational isogenies}
\label{sub:ratl-invariants}

Our algorithm will enumerate period matrices~$\gamma\tau$ using \cref{prop:hecke} and
evaluate the modular forms~$M_k$ at these points. If~$\gamma\tau$ corresponds to an abelian surface~$A'$ that is isogenous
to~$A$ over~$\Q$,
then~$\paren[\big]{M_4(\gamma\tau):M_6(\gamma\tau):M_{10}(\gamma\tau):M_{12}(\gamma\tau)}$
must be rational as a weighted projective point, i.e.~up to complex scaling
with the correct weights. In fact, the converse statement also holds.

\begin{prop}
  \label{prop:ratl-invariants}
  Let~$A$ be a typical p.p.~abelian surface over~$\Q$,
  let~$\tau$ be a period matrix of~$A$, let~$\ell$ be a prime number, and let~$i\in \{1,2\}$. Then
  \begin{enumerate}
  \item
    \label{item:ratl-invariants-distinct}
    As~$\gamma$ runs through~$S(\ell^i)$, the projective
    points~$\paren[\big]{M_4(\gamma\tau):\cdots:M_{12}(\gamma\tau)}$
    in~$\P^{4,6,10,12}(\C)$ are all distinct.
  \item \label{it:ratl-invariants-2} Fix~$\gamma\in S(\ell^i)$, and assume that
    there exists a scalar~$\lambda\in \C^\times$ such that
    $\lambda^k M_k(\gamma\tau)\in\Q$ for each
    $k\in\{4,6,10,12\}$. Then~$\gamma\tau$ is a period matrix of a p.p.~abelian
    surface~$A'$, defined over~$\Q$, such that there exists an $i$-step isogeny
    $A\to A'$ of degree~$\ell^{2i}$ defined over~$\Q$.
  \end{enumerate}
\end{prop}

\begin{proof}
  We first prove~\eqref{item:ratl-invariants-distinct} by contradiction.  If
  these projective points happen to be equal for some~$\gamma_1\neq\gamma_2$,
  then we have a (non-commutative) diagram
  \begin{displaymath}
    \begin{tikzcd}
      & A \ar[ld, swap, "f_1"] \ar[rd, "f_2"] \ar[rr, "\lambda_A"] && A^\vee\\
      A_1 \arrow[rr, "\eta", "\sim"'] && A_2 \ar[rr, "\lambda_{A_2}"] &&
      A_2^\vee \ar[ul, swap, "f_2^\vee"]
    \end{tikzcd}
  \end{displaymath}
  where~$f_1,f_2$ are isogenies with distinct kernels, and~$\eta$ is an
  isomorphism. Then the compositions
  $\eta\circ f_1\circ \lambda_A^{-1}\circ \dual{f_2}\circ \lambda_{A_2}$ and
  $f_2\circ \lambda_A^{-1}\circ \dual{f_2}\circ \lambda_{A_2}$ yield elements
  of~$\End((A_2)_{\Qbar}) = \Z$ of the same degree~$\ell^{2i}$. They must
  therefore be equal up to sign.  Thus $\eta \circ f_1 = \pm f_2$ and
  $\ker f_1 = \ker f_2$, a contradiction.

  For~\eqref{it:ratl-invariants-2}, let~$K$ be the subgroup of~$A$ attached
  to~$\gamma$. If~$K$ is not defined over~$\Q$, then the Galois orbit of~$K$
  consists of several subgroups of~$A$, and the associated quotients have the
  same modular invariants,
  contradicting~\eqref{item:ratl-invariants-distinct}. Therefore~$K$ and the
  associated isogeny are defined over~$\Q$.
\end{proof}

\subsection{Certification}
\label{sub:certification}

By \cref{prop:ratl-invariants}, evaluating modular invariants and detecting
when they are rational suffices to detect rational isogenies. In practice
however, we will have to manipulate these complex numbers up to some finite
precision. While rational invariants can still be detected heuristically and
the isogeny might be proved to exist by other methods, certifying the
non-existence of a rational isogeny is not immediate.

A key idea to achieve certification is to leverage the fact that we manipulate
modular forms with integral Fourier coefficients to compute algebraic integers
instead of just complex numbers. This technique is inspired from the description
of denominators for modular equations in ~\cite{KiefferEvaluation}, and allows us to rule out non-rational isogenies. As a byproduct, we are also
able to certify the existence of rational isogenies using computations
over~$\C$ only.

Concretely, let~$A,\tau$ and~$\ell$ be as in \cref{prop:ratl-invariants}, and denote
the modular invariants of~$A$ as
defined in \S \ref{sub:siegel-mf} by~${(m_4,m_6,m_{10},m_{12})\in \Z^4}$. Then there exists a scalar~$\lambda\in \C^\times$,
uniquely determined up to sign (but dependent on~$\tau$), such that
\begin{equation}
\label{eq:lambda-mk}
  \lambda^k M_k(\tau) = m_k \text{ for all } k\in \{4,6,10,12\}.
\end{equation}
Let~$f$ be a Siegel modular form of even
weight~$k$ with integral Fourier coefficients. For~$\gamma\in S(\ell^i)$, we define
\begin{equation}
\label{eq:Nfgamma}
  N(f,\gamma) \coloneqq (12\lambda)^k (\ell^d \det(\gamma^*\tau)^{-1})^k f(\gamma\tau).
\end{equation}
with $d=2$ if $i=1$, and $d=3$ if $i=2$. (The central factor
$\ell^d\det(\gamma^*\tau)^{-1}$ is a power of~$\ell$ independent of~$\tau$.) We
will show that the $N(f,\gamma)$ are in fact algebraic integers.

The first step is to reinterpret the complex numbers $N(f,\gamma)$
algebraically. Let~$\omega$ be a basis of $\wedge^2\Omega^1(A)$ such that
$M_j(A,\omega) = m_j$ for all $j\in\{4,6,10,12\}$, and let~$K$ be the subgroup
of~$A$ corresponding to~$\gamma$ via the correspondence of
\cref{prop:hecke}. The subgroup~$K$ is not necessarily defined over~$\Q$;
let~$L$ be its field of definition, which is a number field embedded
in~$\C$. Then pulling back differential forms along the quotient isogeny
$A\to A/K$ induces a bijection between the $L$-vector spaces
$L\otimes_{\Q}\Omega^1(A)$ and~$\Omega^1(A/K)$. Under this bijection, $\omega$ corresponds to an $L$-basis~$\omega'$ of~$\wedge^2\Omega^1(A/K)$.

\begin{lem}
\label{lem:Nfgamma-algebraic}
    With the above notation, we have $N(f,\gamma) = 12^k f(A/K,\omega')$.
\end{lem}

\begin{proof}
    By~\cite[Rem.~8.1.4]{BirkenhakeLange}, for every $\tau'\in\Half_2$ and $\gamma\in \Sp(4,\Z)$, the map $z\mapsto (\gamma^*\tau')^{-t}z$ defines an isomorphism between the abelian surfaces $\C^2/(\Z^2\oplus\tau'\Z^2)$ and $\C^2/(\Z^2\oplus(\gamma\tau')\Z^2)$. Keeping the notation from above the lemma, we
  write $\gamma = \delta_2 \Delta \delta_1$ where
  $\delta_1,\delta_2\in \Sp(4,\Z)$ and $\Delta$ is either
  $\Diag(1,1,\ell,\ell)$ or $\Diag(1,\ell,\ell^2,\ell)$ depending
  on~$i$. We then consider the commutative diagram
  \[
  \begin{tikzcd}[sep=tiny]
      & A \ar[rrr,  ->>]  \ar[ld, swap, "\eta_1"] &&& A/K \ar[rd, "\eta_2"]& \\
      \C^2/(\Z^2\oplus\tau\Z^2) \ar[rd, "\zeta_1"] & &&& & \C^2/(\Z^2\oplus\gamma\tau\Z^2) \\
      &
      \C^2/(\Z^2\oplus\delta_1\tau\Z^2)
      \ar[rrr, "\xi", ->>] &&&
      \C^2/(\Z^2\oplus \Delta \delta_1\tau\Z^2) \ar[ru, "\zeta_2"] &
  \end{tikzcd}
  \]
  where~$\zeta_1$ (resp.~$\zeta_2$) is the isomorphism $z\mapsto (\delta_1^*\tau)^{-t}z$ (resp.~$z\mapsto (\delta_2^*(\Delta\delta_1\tau))^{-t} z$),
  double-tipped arrows denote
  natural quotient maps, $\eta_1$ is an isomorphism of complex tori,
  and~$\eta_2$ is another isomorphism that is determined by the choice
  of~$\eta_1$. We recall that $\omega(\tau)$ denotes the element
  $2\pi i dz_1 \wedge 2\pi i dz_2$
  of~$\wedge^2\Omega^1\paren[\big]{\C^2/(\Z^2\oplus\tau\Z^2)}$, where~$z_1,z_2$
  are the coordinates of~$\C^2$.  Let~$r\in \C^\times$ be such that
  $\omega' = r\,\eta_2^*\,\omega(\gamma\tau)$ as differential forms on~$A/K$. We have
  $f(A/K,\omega') = r^{-k} f(\gamma\tau)$. Using the bottom line of the diagram,
  we find that
  \[
    \omega(\tau) = \det(\delta_1^*\tau) \cdot
    \det\paren[\big]{\delta_2^*(\Delta\delta_1\tau)} \cdot
    (\zeta_2\circ\xi\circ\zeta_1)^*\,\omega(\gamma\tau).
    \]
    We can rewrite this equality using the following cocycle relation:
    \[
    \gamma^*\tau = (\delta_2\Delta\delta_1)^*\tau = \delta_2^*(\Delta\delta_1\tau) \cdot \Delta^*(\delta_1\tau) \cdot \delta_1^*\tau.
    \]
    As $\det \Delta^*(\delta_1\tau) = \ell^d$, we have
    \[
    \omega(\tau) = \ell^{-d} \det(\gamma^*\tau) \cdot (\zeta_2\circ\xi\circ\zeta_1)^*\,
    \omega(\gamma\tau)
  \]
  We deduce that
  $\omega = r\, \ell^d \det(\gamma^*\tau)^{-1} \eta_1^*\omega(\tau)$ as
  differential forms on~$A$, and thus
  $\lambda = \pm r^{-1} \ell^{-d} \det(\gamma^*\tau)$. From~\eqref{eq:Nfgamma},
  we finally obtain
  \[
    N(f,\gamma) = 12^k r^{-k} f(\gamma\tau) = 12^k f(A/K,\omega'). \qedhere
  \]
\end{proof}

\begin{thm}
  \label{thm:integers}
  Let~$A$ be a p.p.~abelian surface defined over~$\Q$, let~$\tau\in \Half_2$ be
  a period matrix of~$A$, let~$\ell$ be a prime, and let $i\in\{1,2\}$.  For a
  Siegel modular form~$f$ on~$\Half_2$ of even weight with integral Fourier
  coefficients and~$\gamma\in S(\ell^i)$ (cf.~\S \ref{sub:hecke}), we
  define~$N(f,\gamma)$ as in equation~\eqref{eq:Nfgamma}. Then
  \begin{enumerate}
  \item \label{it:thm-1} For each such modular form~$f$, the
    set~$\{N(f,\gamma): \gamma\in S(\ell^i)\}$ is a Galois-stable set of
    algebraic integers.  Moreover, the action of~$\Gal(\Qbar/\Q)$ on this set
    corresponds to the Galois action on subgroups of~$A[\ell^i]$ via the
    correspondence of \cref{prop:hecke}.
  \item \label{it:thm-2} If~$\gamma\in S(\ell^i)$ corresponds to a subgroup~$K$
    of~$A$ that is defined over~$\Q$, then
    \[
      \paren[\big]{N(M_4,\gamma):N(M_6,\gamma):N(M_{10},\gamma): N(M_{12},\gamma)} \in \P^{4,6,10,12}(\Q)
    \]
    are the modular invariants of the quotient~$A/K$ in the sense of \S
    \ref{sub:siegel-mf}.
  \end{enumerate}
\end{thm}

\begin{proof}
  By \cref{lem:Nfgamma-algebraic}, the complex numbers $N(f,\gamma)$ are in
  fact algebraic, form a Galois-stable set, and the action of $\Gal(\Qbar/\Q)$
  on them corresponds to the Galois action on subgroups of~$A$. We now show
  that the~$N(f,\gamma)$ are algebraic integers. To this end, we construct the
  monic polynomial
  \begin{displaymath}
    P(X) = \prod_{\gamma\in S(\ell^i)} \paren[\big]{X - N(f,\gamma)}.
  \end{displaymath}
  Let $n = \deg(P) = \#S(\ell^i)$.  Expanding the product, we observe as
  in~\cite[Prop.\ 2.4]{KiefferEvaluation} that for each $0\leq j\leq n$, the
  coefficient of~$X^{n-j}$ in $P$ takes the form
  \begin{displaymath}
    (12\lambda)^{kj} \cdot g_j(\tau),
  \end{displaymath}
  where~$g_j$ is a Siegel modular form of weight~$kj$ with integral Fourier
  coefficients.  By \cref{prop:int-gens}, the modular form $12^{kj}g_j$ is an
  element of~$\Z[M_4,M_6,M_{10},M_{12}]$ of weight~$kj$.  Therefore
  \[
    (12\lambda)^{kj}g_j(\tau) \in \Z[m_4,m_6,m_{10},m_{12}] = \Z.
  \]
  Thus~$P$ has integral coefficients, which
  concludes the proof of~\eqref{it:thm-1}.

  Item~\eqref{it:thm-2} is a direct consequence of
  \cref{lem:Nfgamma-algebraic}: we showed the existence of a basis $\omega$ of
  the $\Q$-vector space $\wedge^2\Omega^1(A/K)$ such that
  $N(M_k,\gamma) = 12^k M_k(A/K,\omega)$ for all $k\in \{4,6,10,12\}$, and
  absorbing the scalar factors $12^k$ does not modify the projective point.
\end{proof}

\subsection{Outline of the algorithm}
\label{sub:outline}

The theoretical background being set, we describe the concrete computations
over~$\C$ that we perform to list the rational 1- or 2-step isogenies from a
given typical p.p.~abelian surface~$A$ over~$\Q$. We take the modular
invariants of~$A$ as input. If~$A$ is given as the Jacobian of a genus~$2$
curve~$C$ over~$\Q$ as in \cref{alg:isog-class}, then its modular invariants
can be determined by polynomial formulas in terms of the coefficients of~$C$,
as explained in~\S\ref{sub:siegel-mf}. Similarly, the output will consists of
modular invariants of isogenous abelian surfaces: we refer to \cref{sec:curves}
for the subsequent reconstruction of genus~$2$ curve equations.

In the whole algorithm, we use interval arithmetic to keep track of precision
losses: each complex number is encoded as a ball that provably
contains the exact value. Of course, we need to assume more than mere
correctness of the underlying arithmetic for our algorithms to work, as we will
not get far if every operation returns the whole of~$\C$. A minimal assumption
in this subsection is that any given computation with exact input will output a
complex ball whose radius tends to zero as its \emph{working precision} tends
to infinity. This will ensure that our algorithms terminate. In fact, we have a
precise control on the precision losses incurred during the computations: we
postpone this discussion to the complexity analysis
in~\S\ref{sub:implementation}.

There are two main stages in the algorithm. First, we use low-precision
computations to filter out symplectic matrices~$\gamma\in S(\ell)$
or~$S(\ell^2)$ whose associated~$N(f,\gamma)$, seen as a ball, does not contain
any rational integer. By \cref{thm:integers}, these matrices~$\gamma$ do not
correspond to rational isogenies with domain~$A$. One expects that all
remaining subgroups correspond to rational isogenies. Second, we use
high-precision computations to either certify or disprove that the remaining
subgroups are defined over~$\Q$. We will use two black boxes, namely the
computation of a period matrix from modular invariants and the evaluation the
Siegel modular forms~$M_k$ at a point of~$\Half_g$; we also refer
to~\S\ref{sub:implementation} for a discussion of these steps. In the first
stage, we proceed as follows.

\begin{algorithm} \label{alg:cplx-1}
  {\em Input:}
  \begin{itemize}
  \item the modular invariants $m_4,m_6,m_{10},m_{12}$ of a p.p.~abelian surface~$A$ over~$\Q$,
  \item a prime $\ell$, and $i\in\{1,2\}$ indicating either 1-step
    or 2-step isogenies.
  \end{itemize}
  {\em Output:}
  \begin{itemize}
  \item a period matrix~$\tau$ of~$A$ (to low precision),
  \item a list~$L$ of symplectic matrices that provably contains
    all~$\gamma\in S(\ell^i)$ associated with a rational subgroup
    of~$A[\ell^i]$ for this choice of~$\tau$,
  \item for each~$\gamma\in L$ and $k\in\{4,6,10,12\}$, a candidate value
    $m_k'(\gamma)\in \Z$ for the algebraic integer $N(M_k,\gamma)$ satisfying $\abs{N(M_k,\gamma) - m_k'(\gamma)}\leq 2^{-10}$;
  \item for each $\gamma\notin L$ and $k\in\{4,6,10,12\}$, a
    ball containing $N(M_k,\gamma)$ of radius at most $2^{-10}$.
  \end{itemize}
  \begin{enumerate}[label={\bf Step \arabic*.},labelwidth=\widthof{{\bf Step 2.}},leftmargin=!]
  \item Compute a low-precision approximation of a period matrix~$\tau$ of~$A$, see Theorem \ref{thm:period-matrix}.
  \item Evaluate the modular forms~$M_4,\ldots,M_{12}$ at~$\tau$, and deduce an
    approximation of the scaling factor~$\lambda$ as defined
    in~\S\ref{sub:certification}.
  \item For each~$\gamma\in S(\ell^i)$ and each~$k\in \{4,6,10,12\}$,
    evaluate~$N(M_k, \gamma)$. If the radius of one of these approximations is
    larger than~$2^{-10}$, double the working precision and restart.
  \item Let~$L$ be the list of all~$\gamma$ such that each of the
    balls~$N(M_k,\gamma)$ for $k\in \{4,6,10,12\}$ contains an integer
    $m_k'(\gamma)$, and return the output listed above.
  \end{enumerate}
\end{algorithm}

At the end of \cref{alg:cplx-1}, for each~$\gamma\in L$, the
integers~$m_k'(\gamma)$ are the only possible values for the modular invariants
of the isogenous abelian surface if it is indeed defined over~$\Q$. One
could stop here and certify the existence of an isogeny by other methods. In
order to certify that these likely isogenies indeed exist using computations
over~$\C$ only, we proceed as follows.

\begin{lem}
  \label{lem:certification}
  Keep the notation of \cref{alg:cplx-1}. Let~$k\in \{4,6,10,12\}$, let
  $m\in \Z$, and let $S_0\subset S(\ell^i)$ be a nonempty subset such that
  $\abs{N(M_k,\gamma)- m} \leq 1$ for all~$\gamma\in S_0$, and
  $N(M_k,\gamma)\neq m$ for all $\gamma\in S(\ell^i)\backslash S_0$. Let
  \[
    B = \prod_{\gamma\in S(\ell^i)\backslash S_0} \abs{N(M_k,\gamma) - m}.
  \]
  Then for each $\gamma\in S_0$, the inequality
  $\abs{N(M_k,\gamma) - m} < \displaystyle\smash{\frac{1}{B}}$ holds if and only if
  $N(M_k,\gamma) = m$.
\end{lem}

\begin{proof}
  Let~$S_1\subset S_0$ be the subset of matrices~$\gamma$ such that the
  equality $N(M_k,\gamma) = m$ holds, and consider the polynomial
  \[
    P = \prod_{\gamma\in S(\ell^i)} \paren[\big]{X - N(M_k,\gamma) + m} \in
    \Z[X], \quad \text{by \cref{thm:integers}}.
  \]
  The coefficient of $X^j$ in~$P$ is zero for $0\leq j\leq \#S_1 - 1$, and the
  coefficient of $X^{\#S_1}$ is, up to sign,
  \[
    \prod_{\gamma\in S(\ell^i)\backslash S_0} \paren[\big]{N(M_k,\gamma) - m} \cdot
    \prod_{\gamma\in S_0\backslash S_1} \paren[\big]{N(M_k,\gamma) - m}.
  \]
  By construction, this integer is nonzero, hence at least~$1$ in absolute
  value. Thus, for every~$\gamma\in S_0\backslash S_1$, we have as claimed
  \[
    \abs{N(M_k,\gamma)- m} \geq \frac1B.
    \qedhere
  \]
\end{proof}

\begin{algorithm}
  \label{alg:cplx-2}
  {\em Input:}
  the input and output of~\cref{alg:cplx-1}. \\
  {\em Output:} the modular invariants of all the p.p.~abelian surfaces linked
  to~$A$ by an $i$-step $\ell$-isogeny, as a list of integer-valued
  tuples~$(m_4',m_6',m_{10}',m_{12}')$.
  \begin{enumerate}[label={\bf Step \arabic*.},labelwidth=\widthof{{\bf Step 2.}},leftmargin=!]
  \item For each~$\gamma_0\in L$ and $k\in \{4,6,10,12\}$, let~$S_0\subset L$
    to be the set of all matrices~$\gamma$ such that
    $m_k'(\gamma) = m_k'(\gamma_0)$, and compute a low-precision upper bound
    for the quantity~$B$ as in \cref{lem:certification}. Let~$B_0$ be their
    maximum value ranging over all~$\gamma_0$ and~$k$.
  \item Choose a higher working precision (more on this below). Recompute the
    period matrix~$\tau$, as well as~$N(M_k,\gamma)$ for each~$\gamma\in L$
    and~$k\in\{4,6,10,12\}$.
  \item For each~$\gamma\in L$ and $k\in \{4,6,10,12\}$, check
    whether~$N(M_k,\gamma)$ still contains the candidate value~$m_k'(\gamma)$.
    If not, remove~$\gamma$ from~$L$. If yes, check whether the inequality
    \[
      \abs{N(M_k,\gamma) - m_k'(\gamma)} < \frac{1}{B_0}
    \]
    holds; if this cannot be decided, double the working precision and go back
    to Step~2.
  \item Output the list of tuples $(m_4'(\gamma),\ldots,m_{12}'(\gamma))$ for
    the remaining~$\gamma\in L$.
  \end{enumerate}
\end{algorithm}

As soon as all the radii of the balls containing $N(M_k,\gamma)$ are less than
$\frac{1}{2B_0}$, there will be no further doubling of the working precision
in~Step~3, thus \cref{alg:cplx-2} terminates. \Cref{lem:certification}
guarantees that the algorithm is correct.

\begin{remark}
  In practice, increasing the working precision by~$p$ bits results in output
  intervals whose radius is multiplied by roughly $2^{-p}$. Therefore a good
  guess for the choice of high precision in Step~2 of \cref{alg:cplx-2} is to
  add $\ceil{\log_2(B_0)}$ bits to the current working precision at the end of
  \cref{alg:cplx-1}. We then expect that no further precision increases are
  needed in Step~3.
\end{remark}

\subsection{Implementation details and complexity analysis.}
\label{sub:implementation}

We conclude this section with a complexity analysis of the isogeny algorithm in
terms of~$\ell$,~$i$ and the height of the integers~$m_k$ for
$k\in\{4,6,10,12\}$. First, we give an asymptotic upper bound on the
absolute values $\abs{N(M_k,\gamma)}$: this will specify the necessary working
precisions in \cref{alg:cplx-1,alg:cplx-2}. Then, we review previous works on
the computation of period matrices and the evaluation of modular forms in
quasi-linear time in terms of the required precision, filling in the two black
boxes of \S\ref{sub:outline}.

We may fix a compact subset $\mathcal{G}\subset \Half_2\cup\{\infty\}$, where
$\infty$ denotes the cusp at infinity, such that the period matrix~$\tau$ in
Step~1 of \cref{alg:cplx-1} always belongs to~$\mathcal{G}$. For instance, we
can take~$\mathcal{G}$ to be the set of points at a distance at most~$2^{-10}$
from the Siegel fundamental domain~$\mathcal{F}$ (see \cite[\S6.2]{Streng}),
plus the point~$\infty$.  Recall that the modular forms~$M_k$ are bounded in a
neighborhood of~$\infty$, never vanish simultaneously on~$\Half_2$, and
that~$M_4$ and~$M_6$ take the value~$1$ at~$\infty$. Thus there exist two
absolute constants $0 < C_1 < C_2 < +\infty$ such that $M_k\leq C_2$ uniformly
on~$\mathcal{G}$ for each~$k\in\{4,6,10,12\}$, and
$\max\bigl\{\abs{M_k(\tau)}: k\in \{4,6,10,12\}\bigr\} \geq C_1$ for
each~$\tau\in \mathcal{G}$.

\begin{lem}
  \label{lem:size-estimates}
  In \cref{alg:cplx-1}, assume that~$\tau\in \mathcal{G}$, and let
  \[
  h = \log\max\{\abs{m_4},\ldots, \abs{m_{12}}\}.
  \]
  Then in Step~2 of that
  algorithm, we have $\log\,\abs{\lambda} = O(h)$. In Step~3, we have
  $\log\,\abs{N(M_k,\gamma)} = O(h + \log\ell)$ for each~$\gamma\in S(\ell^i)$
  and $k\in \{4,6,10,12\}$. Both implied constants are absolute.
\end{lem}

\begin{proof}
  For the first part, we have $\log\,\abs{\lambda} \leq \frac14(h - \log
  C_1)$. For the second part, we need to estimate~$\abs{M_k(\gamma\tau)}$. The
  matrix~$\gamma\tau\in \Half_2$ will not usually belong to~$\mathcal{G}$, but
  its imaginary part is nevertheless ``bounded below'': if~$a,b,0,d$ denote the
  $2\times 2$ blocks of~$\gamma$, we have
  $\im(\gamma\tau) = a \im(\tau) d^{-1}$, so
  $\det \im(\gamma\tau) \geq C_3/\ell^2$ where~$C_3$ is an absolute constant. Let
  now~$\eta\in \Sp(4,\Z)$ be such that~$\eta\gamma\tau\in \mathcal{G}$, and
  let~$r = \det(\eta^*(\gamma\tau))$. Then by \cite[Proof of
  Prop.~1.1]{Klingen}, we have
  \[
    \det(\im(\eta\gamma\tau)) = \abs{r}^{-2}\det(\im(\gamma\tau))
  \]
  so $\abs{r}^{-2} \leq C_4 \ell^2$, where~$C_4$ is an absolute constant. Finally we have
  \[
    \log\,\abs{M_k(\gamma\tau)} = \log\,\abs{r^{-k} M_k(\eta\gamma\tau)} = O(\log\ell)
  \]
  where the implied constant is absolute.
\end{proof}

With these estimates in hand, we focus on the two key computations left aside
in~\S\ref{sub:outline}. From the data of modular invariants in~$\Z$, a
corresponding period matrix~$\tau\in\mathcal{F}$ can be efficiently computed
using arithmetic-geometric means, an algorithm described
in~\cite[Chap.~9]{Dupont} and proved correct in~\cite{KiefferSignChoices}. (In
the first low-precision step, we could also use numerical
integration~\cite{MolinNeurohr}, but the dependency of this algorithm on~$h$
has not been made explicit.)

\begin{thm}[{\cite[\S4.1]{KiefferEvaluation}}]
  \label{thm:period-matrix}
  There exists an algorithm
  which, given the modular invariants $m_4,\ldots,m_{12}\in\Z$ of a typical
  p.p.~abelian surface~$A$ over~$\Q$ and $p\geq 1$, computes an approximation
  of a period matrix~$\tau\in\mathcal{F}$ of~$A$ up to an error of~$2^{-p}$
  within a running time of $\Otilde(p+h)$ binary operations, where
  $h = \log\max\{\abs{m_4},\ldots, \abs{m_{12}}\}$.
\end{thm}

In order to evaluate the modular forms~$M_k$ at a
point~$\gamma\tau\in \Half_2$, we compute a matrix~$\eta\in\Sp(4,\Z)$ such
that~$\eta\gamma\tau$ is close to the fundamental domain~$\mathcal{F}$ and
evaluate $M_k(\eta\gamma\tau)$. The latter can be rewritten as polynomials in
terms of theta constants at~$\eta\gamma\tau$: see~\cite[\S7.1]{Streng} for
explicit formulas. Following \cite[\S4.3]{KiefferEvaluation}, we can control the
cost of the reduction step in terms of the quantity
\[
  \Lambda(\gamma\tau) = \log\max\{2, \abs{\gamma\tau},
  \det(\im(\gamma\tau))^{-1}\},
\]
where~$\abs{\gamma\tau}$ denotes the largest absolute value of an entry
of~$\gamma\tau$. As a consequence of~\cite[Cor.~7.8]{Streng}, we
have~$\abs{\tau} = O(h)$, so~$\Lambda(\gamma\tau) = O(\log h + \log\ell)$ where
the implied constant is absolute. After the reduction step, the evaluation of
theta constants on~$\mathcal{F}$ can be done in uniform quasi-linear time using
arithmetic-geometric means and Newton's method. The following result summarizes
this approach.

\begin{thm}[{\cite[\S4.3]{KiefferEvaluation}}]
  \label{thm:theta}
  There exists an algorithm and an
  absolute constant~$C_5$ such that the following holds. Let~$\tau\in\Half_2$
  and~$p\geq 1$. Then, given an approximation of~$\tau$ to
  precision~$p+C_5\Lambda(\tau)$, the algorithm computes a
  matrix~$\gamma\in \Sp(4,\Z)$ such that
  $\log\,\abs{\gamma} = O(\Lambda(\tau))$, a matrix~$\tau'\in \mathcal{F}$ such
  that~$\abs{\tau'-\gamma\tau}\leq 2^{-p}$, and an approximation of the squares
  of theta constants at~$\gamma\tau$ up to an error of~$2^{-p}$. Its running
  time is $\Otilde(\Lambda(\tau)^2 + p)$ binary operations.
\end{thm}

We can now prove the complexity bound stated at the beginning of
\cref{sec:analytic}.

\begin{cor}
  \label{cor:complexity}
  Let~$A$ be a typical p.p.~abelian surface over~$\Q$ with modular invariants
  $m_4,\ldots,m_{12}$, and let
  $h = \log\max\{\abs{m_4},\ldots,\abs{m_{12}}\}$. Let~$\ell$ be a prime, let~$i\in\{1,2\}$, and let~$n$ be the size of
  the list~$L$ returned by \cref{alg:cplx-1} on this input. Then one can run
  \cref{alg:cplx-1,alg:cplx-2} to detect $i$-step $\ell$-isogenies with
  domain~$A$ using a total of $\Otilde\paren[\big]{(n+1)\ell^d h}$ binary
  operations, where~$d=3$ when~$i=1$ and~$d=4$ when~$i=2$.
\end{cor}

\begin{proof}
  By \cref{lem:size-estimates}, in Step~3 of \cref{alg:cplx-1}, we must
  evaluate~$M_k(\gamma\tau)$ to $O(h+\log\ell)$ bits of precision for
  each~$\gamma\in S(\ell^i)$. By \cref{thm:period-matrix,thm:theta}, this can
  be done within $\Otilde(\ell^dh)$ binary operations, as
  $\#S(\ell^i) = O(\ell^d)$. Similarly, in \cref{alg:cplx-2}, we have
  $\log\,\abs{B_0} = O(\ell^d(h+\log\ell))$, so we need to evaluate
  $M_k(\gamma\tau)$ to $O\paren[\big]{\ell^d(h+\log\ell)}$ bits of precision
  for each~$\gamma\in L$. This can be done in $\smash{\Otilde}(n\ell^d h)$ binary
  operations.
\end{proof}

\section{Reconstructing genus~2 curves}
\label{sec:curves}

\subsection{Mestre's algorithm}
\label{sec:Mestre}

Given a quadruple of Igusa--Clebsch invariants $I = (I_2 : I_4 : I_6 : I_{10})$ over~$\Q$,
there always exists a genus 2 curve over $\Qbar$ having these invariants, but
it cannot always be defined over~$\Q$. This phenomenon is known as Mestre's
obstruction. Given~$I$, Mestre~\cite{Mestre} constructs a conic $L$ over~$\Q$
together with an effective divisor $D$ of degree 6 on this conic with the
following properties. If $L(\Q) \neq \emptyset$ and hence $L \cong \P^1$, then
there exists a genus 2 curve over $\Q$ with invariants $I$ whose Weierstrass
points correspond to the points in~$D$. If $L(\Q) = \emptyset$, there is no
genus 2 curve over $\Q$ having these invariants.  This reconstruction algorithm
has been implemented in many computer algebra systems and the reader is
referred to \cite{Mestre} for more details about the method.

In our setting, Mestre's obstruction does not arise, and $L(\Q)$ is always
non-empty. Indeed, by \cref{prop:ratl-invariants}, we only manipulate typical
p.p.~abelian surfaces defined over~$\Q$. By \cite[Appendix, Thm.\
4]{SerreAppendix}, these surfaces all are Jacobians of genus~2 curves defined
over~$\Q$.

Mestre's algorithm usually produces curves with very large coefficients. Their
sizes can be reduced by applying \cite{StollCremona}, which has been
implemented in \cite{PARI} and \cite{Magma}. Note that even this reduced model
is not necessarily unique.

\subsection{Identifying the correct twists}
\label{sub:twists}

Having constructed a curve~$C'$ over~$\Q$ from the invariants output by
\Cref{alg:cplx-2}, we have no guarantee yet that $\Jac(C)$ will be isogenous
to~$\Jac(C')$ over~$\Q$: we can only say that $\Jac(C')$ is a twist of the
abelian surface isogenous to~$\Jac(C)$.  Since $\End(\Jac(C')_{\Qbar}) = \Z$,
the only possible twists of~$\Jac(C')$ are quadratic twists, and correspond to
quadratic twists of the curve~$C'$ itself. Therefore, there will be a unique
twist~$C''$ of~$C'$ (up to isomorphism over $\Q$) such that $\Jac(C)$ is
isogenous to~$\Jac(C'')$.

We use the following method to find~$C''$. For a genus 2 curve $C$ and a prime
of good reduction~$\ell$ of~$\Jac(C)$, we denote by~$a_{\ell}(C)\in\Z$ the
trace of Frobenius on the reduction on~$\Jac(C)$ modulo~$\ell$.

\begin{algorithm}\label{alg:twists}
  {\em Input:} curves $C$ and $C'$ of genus 2 over $\Q$ with typical Jacobians,
  such that some twist of $\Jac(C')$ is isogenous to $\Jac(C)$.\\
  {\em Output:} the unique twist $C''$ of $C'$ such that $\Jac(C)$ is isogenous
  to~$\Jac(C'')$.
\begin{enumerate}[label={\bf Step \arabic*.},labelwidth=\widthof{{\bf Step 2.}},leftmargin=!]
\item Compute a set $\mathcal{B}$ of primes containing the bad primes of $C$
  and $C'$. Let
  $$G = \langle -1 \rangle \times \langle b : b \in \mathcal{B} \rangle \subset
  \Q^* / \Q^{*2}.$$
\item Find auxiliary primes $\ell_1, \ldots, \ell_k$ such that the Frobenius
  traces $a_{\ell_i}(C)$ are nonzero and the map
  $$\mu\colon G \to \{ \pm 1 \}^k,\quad  x \mapsto
  \left(\genfrac(){}{0}{x}{\ell_i}\right)_{1\leq i\leq k}$$ is injective.
\item Identify the unique element $g \in G$ for which the twist $C''$ of $C'$
  by $g$ satisfies $a_{\ell_i}(C'') = a_{\ell_i}(C)$ for every~$1\leq i\leq k$,
  and return~$C''$.
\end{enumerate}
\end{algorithm}

\begin{prop}
\cref{alg:twists} terminates and is correct.
\end{prop}

\begin{proof}
  First, we will prove the existence of the auxiliary primes
  $\ell_1, \ldots, \ell_k$. By \cite[Theorem 1]{LangTrotterResult} and the fact
  that $\Jac(C)$ is typical, the equality $a_{\ell}(C) = 0$ holds only for a
  density~0 subset of primes.  Write
  $\mathcal{B} \cup \{-1\} = \{b_1,\ldots,b_k\}$. For each~$1\leq i\leq k$, by
  quadratic reciprocity, we can find congruence conditions on~$\ell$
  guaranteeing that $\genfrac(){}{}{b_i}{\ell} = -1$ and
  $\genfrac(){}{}{b_j}{\ell} = 1$ for each~$j\neq i$. By Dirichlet's prime
  number theorem and the fact that only a density 0 subset of the primes are
  excluded, there exists a prime $\ell_i$ (in fact infinitely many) satisfying
  these congruence conditions and such that $a_{\ell_i}(C) \ne 0$. With this
  choice of $\ell_1, \ldots, \ell_k$, the map~$\mu$ is injective.

  Next, we recall the Néron--Ogg--Shafarevich criterion~\cite{SerreTate}: an
  abelian surface~$A$ has good reduction at a prime~$p$ if and only if for all
  primes $\ell \ne p$, the Galois representation $T_{\ell}(A)$ is unramified
  at~$p$. Hence, if $A$ has good reduction at $p$ and $A'$ is the quadratic
  twist of $A$ by a squarefree integer~$D$ divisible by~$p$, then $A'$ has bad
  reduction at $p$. Indeed, the Galois representation $T_{\ell}(A')$ is
  obtained from $T_{\ell}(A)$ by tensoring with the quadratic character
  associated with $\Q(\sqrt{D})$, which is ramified at $p$. On the other hand,
  any abelian surface isogenous to~$A$ must have the same primes of bad
  reduction.

  As a consequence, the correct twist $\Jac(C'')$ of $\Jac(C')$ is given by a
  squarefree integer~$D$ that can only be divisible by primes of bad reduction
  of $\Jac(C)$ or $\Jac(C')$. These form a subset of $\mathcal{B}$, so the
  correct twist is among those enumerated by $G$.  By the choice of the
  auxiliary primes $\ell_1, \ldots, \ell_k$, the tuples of Frobenius traces
  $(a_{\ell_1}, \ldots, a_{\ell_k})$ take distinct values for all the twists
  enumerated by~$G$, so there is a unique output in Step~3.
\end{proof}

\Cref{alg:twists} has an exponential complexity in terms of the number of bad
primes of $C$ and $C'$, but this number is small in practice.

\begin{remark}
  In the case of elliptic curves, an analogue of \cref{alg:twists} would not be
  needed. Indeed, if~$E$ is an elliptic curve over~$\Q$ with automorphism
  group~$\{\pm 1\}$, then the modular invariants of~$E$ in~$\P^{4,6}(\Q)$ determine the $\Q$-isomorphism class of~$E$, as
  twisting~$E$ by~$d$ multiplies its invariants by~$d^2$ and~$d^3$. Thus we
  would obtain the correct twist as a direct result of the computations
  over~$\C$.

  In genus~$2$ however, quadratic twists have the same modular invariants: this
  comes from the fact that twisting an abelian surface~$A$ by~$d\in \Q^\times$ acts on
  $\Omega^1(A)$ as~$\Diag(\sqrt{d},\sqrt{d})$, hence on~$\wedge^2\Omega^1(A)$
  as multiplication by~$d$, which is a rational number. Nevertheless, one would
  still be able to compute the correct twist directly (and thus circumvent
  \cref{alg:twists}) by considering vector-valued Siegel modular forms, or
  equivalently by keeping track of big period matrices. This goes beyond the
  scope of this paper, and \cref{alg:twists} was sufficient for our
  experiments.
\end{remark}

\section{Examples}
\label{sec:examples}

We now give explicit illustrations of the methods developed above. First, we
discuss an example of a 1-step $31$-isogeny; then, we report on the results of
running our algorithm on a large dataset of Jacobians of genus 2 curves that
includes the current LMFDB data~\cite{LMFDB}. The prime $\ell=31$ is the
largest prime for which the Galois representation on~$A[\ell]$ is not
surjective among all the abelian surfaces~$A$ in this dataset~\cite{galreps}.

All computations were ran on a server with an
AMD EPYC 7713 2GHz CPU 
and Sage 9.7 \cite{Sage}, Magma 2.28-2 \cite{Magma}, GP/PARI 2.15.0 \cite{PARI}, and \texttt{pydhme v0.0.6} installed.\footnote{A Sage interface for \cite{hdme} is available at \url{https://github.com/edgarcosta/pyhdme/}.}

\subsection{A 1-step 31-isogeny}
\label{sec:31isog}

Consider the hyperelliptic curve
\[
  C \colon y^2 + (x + 1)y = x^5 + 23x^4 - 48x^3 + 85x^2 - 69x + 45.
\]
The conductor of $\Jac(C)$ is $7^2 \cdot 31^2$.  Combining this with the study
of local Euler factors at $p \in \{3, 5, 11\}$, we see that
$\calL_1 = \emptyset$ and $\calL_2 \subseteq \{31\}$ in the notation of~\S
\ref{sec:subgroups}.

The Igusa--Clebsch invariants $(I_2:I_4:I_6:I_{10})$ of $C$ in $\P^{2,4,6,10}(\Q)$ are
\begin{equation*}
(-324608: 7340502400: -589129410429504: 5306537926135312384),
\end{equation*}
from which we can deduce that $\Jac(C)$ has the following modular invariants:
\begin{equation*}
  (m_4, m_6, m_{10}, m_{12}) = (1909600, 2582145496, 45252529, -59231181184).
\end{equation*}
We apply \cref{alg:cplx-1} with $\ell = 31$ and $i = 1$.  After computing a
period matrix
\[
    \tau \approx
    \begin{pmatrix}
    1.69708 i
    &
    0.31188
    +
    0.84854 i
    \\
    0.31188
    +
    0.84854 i
    &
   -0.18812
    +
    2.09922 i
    \end{pmatrix}
\]
in the Siegel fundamental domain with 300 bits of precision, we conclude that there is a unique coset representative
\begin{equation*}
\gamma = \left(\begin{array}{rrrr}
1 & 0 & 0 & 7 \\
0 & 1 & 7 & 23 \\
0 & 0 & 31 & 0 \\
0 & 0 & 0 & 31
\end{array}\right) \in S(31)
\end{equation*}
such that $N(M_4, \gamma)$ contains an integer.
Indeed, we have
\begin{equation*}
  \begin{aligned}
&|N(M_{4}, \gamma) - \alpha^2 \cdot 318972640| < 7.8 \times 10^{-47}\\
&|N(M_{6}, \gamma) - \alpha^3 \cdot 1225361851336| < 5.5 \times 10^{-39}\\
&|N(M_{10}, \gamma) - \alpha^5 \cdot 10241530643525839| < 1.6 \times 10^{-29}\\
&|N(M_{12}, \gamma) + \alpha^6 \cdot 307105165233242232724| < 4.6 \times 10^{-22}\\
  \end{aligned}
\end{equation*}
where $\alpha = 2^{2} \cdot 3^{2} \cdot 31$.
We then employ \cref{alg:cplx-2}, working with \numprint{4128800} bits of precision, to certify that an abelian surface with projective invariants
 $$
 { (318972640,1225361851336,10241530643525839,-307105165233242232724)}$$
is indeed isogenous to $\Jac(C)$ via a 1-step $31$-isogeny.

Mestre's algorithm and a reduction step (\S\ref{sec:Mestre}) yield the hyperelliptic curve
\small
\begin{equation*}
  y^2 = -1624248 x^6 + 5412412 x^5 - 6032781 x^4 + 876836 x^3 - 1229044 x^2 - 5289572 x - 1087304.
  \label{eqn:Cprime}
\end{equation*}
\normalsize
Applying \cref{alg:twists} we learn that the desired curve~$C''$ is the
quadratic twist of this hyperelliptic curve by $-83761$, and is given by the
equation
\small
\begin{equation*}
  y^2 + x y = -x^5 + 2573 x^4 + 92187 x^3 + 2161654285 x^2 + 406259311249 x + 93951289752862
\end{equation*}
\normalsize
with discriminant $7^{2} \cdot 31^{3} \cdot 83761^{12}$.

The overall computation took
175 
minutes of CPU time and used 6.5 gigabytes of ram, of which roughly 90\% is spent certifying the existence of an isogeny between these two curves, i.e.~in \cref{alg:cplx-2}.

Given~$C$ and~$C''$, we can also independently produce a certificate for the
existence of an isogeny of the correct degree~\cite{rigendos}.  It took about
6.5 CPU hours to produce the 2.8 megabyte certificate.

\subsection{LMFDB and beyond}

We now report on the application of our algorithms to a dataset of
\numprint{1743737} genus 2 curves with trivial geometric endomorphism algebra.
These are all the typical abelian surfaces in a dataset of approximately 5
million curves with conductor up to $2^{20}$ provided to us by Andrew
Sutherland \cite{drewsdatabases}. This dataset expands the current set of genus
2 curves in the $L$-functions and modular forms database (LMFDB) \cite{LMFDB}.

The \numprint{1743737} curves are split among \numprint{1440894} isogeny
classes, while the LMFDB subset contains \numprint{63107} curves split among
\numprint{62600} isogeny classes. These isogeny classes have been identified
using Frobenius traces only \cite[\S 4.3]{genus2database}, a heuristic method
whose results are confirmed by our computations.

We have applied our algorithms to one curve per isogeny class and found
\numprint{600948} new curves in total. \Cref{table:degrees} lists the degrees
of the irreducible isogenies that we found (in other words we ignore 2-step
isogenies arising as the composition of two rational 1-step isogenies), and
\cref{table:classes size} shows the sizes of the \numprint{1440894} isogeny
classes.  Most of the large classes only feature Richelot isogenies,
i.e.~1-step 2-isogenies.
In total, however, only \numprint{242442} of the \numprint{600948} new curves
can be reached from the original dataset via Richelot isogenies.

\renewcommand{\arraystretch}{1.1}
\begin{table}[!ht]
\begin{tabular}{c|ccc|c}
  $d$ & \begin{tabular}{@{}c@{}}Number of isogenies\\ of degree $d$\end{tabular}
      &&
  $d$ & \begin{tabular}{@{}c@{}}Number of isogenies\\ of degree $d$\end{tabular}
  \\
\cline{1-2}\cline{4-5}
$2^{2}$ & \numprint{419157}
 && 
$7^{4}$ & \numprint{246}\\
$2^{4}$ & \numprint{693519}
 && 
$11^{4}$ & \numprint{9}\\
$3^{2}$ & \numprint{11568}
 && 
$13^{2}$ & \numprint{20}\\
$3^{4}$ & \numprint{29742}
 && 
$13^{4}$ & \numprint{9}\\
$5^{2}$ & \numprint{415}
 && 
$17^{2}$ & \numprint{4}\\
$5^{4}$ & \numprint{2440}
 && 
$31^{2}$ & \numprint{1}\\
$7^{2}$ & \numprint{154}
\end{tabular}
  \bigskip
\caption{Number of isogenies of each degree in the extended dataset.}
\label{table:degrees}
\end{table}

\vspace{-5mm}

\begin{table}[!ht]
\begin{tabular}{c|ccc|c}
  $k$ & \begin{tabular}{@{}c@{}}Number of isogeny\\ classes of size $k$\end{tabular}
      &&
  $k$ & \begin{tabular}{@{}c@{}}Number of isogeny\\ classes of size $k$\end{tabular}
  \\
\cline{1-2}\cline{4-5}
1 & \numprint{1032456}
 && 
12 & \numprint{52}\\
2 & \numprint{116847}
 && 
14 & \numprint{102}\\
3 & \numprint{197253}
 && 
16 & \numprint{1555}\\
4 & \numprint{54543}
 && 
18 & \numprint{706}\\
5 & \numprint{15547}
 && 
20 & \numprint{120}\\
6 & \numprint{14323}
 && 
22 & \numprint{99}\\
7 & \numprint{430}
 && 
24 & \numprint{6}\\
8 & \numprint{5594}
 && 
28 & \numprint{4}\\
9 & \numprint{35}
 && 
30 & \numprint{8}\\
10 & \numprint{1214}
\end{tabular}
\bigskip
\caption{Distribution of isogeny class sizes in the extended dataset.}
\label{table:classes size}
\end{table}

\vspace{-5mm}

\begin{remark}
  It is worth noting that~\numprint{195806} of the~\numprint{197253} isogeny
  classes of size~$3$ and~\numprint{15523} of the~\numprint{15547} isogeny
  classes of size~$5$ are only made up of $2$-step $2$-isogenies, of degree 16.  The
  isogeny graphs in these two cases are a triangle $\triangle$ and a bowtie
  $\bowtie$ respectively.  Moreover, there is no isogeny class of size~$2$ made
  of a single $2$-step $2$-isogeny.

  These observations can be explained as follows: the existence of a~$2$-step
  $2$-isogeny~$\varphi_1 \colon A_1 \rightarrow A_2$ always implies the
  existence of a triangle consisting of three~$2$-step $2$-isogenies. Indeed,
  assume that~$\ker f$ is generated by~$e_1,2e_2,2f_2$,
  where~$(e_1,e_2,f_1,f_2)$ is a symplectic basis of~$A_1[4]$.  Then the
  subgroup generated by~$e_1+2f_1,2e_2,2f_2$ is another rational maximal
  isotropic subgroup of $A_1[4]$, and gives rise to an
  isogeny~$\varphi_2 \colon A_1 \rightarrow A_3$.  Furthermore, one can show
  that there exists another~$2$-step
  $2$-isogeny~$\varphi_3 \colon A_2 \rightarrow A_3$ such
  that~$\varphi_2 \circ [2] = \varphi_3 \circ \varphi_1$.
\end{remark}

The whole computation took
111 
days of CPU time
and used 215 megabytes of RAM on average per class.
Only 30 classes taking more than 10 minutes.
In these 29 cases, we had to search for and potentially
certify isogenies of large degree.  For 6 of the classes, it took on average 18 minutes to prove the nonexistence of 1 and 2-step $29$-isogenies; one of them
has LMFDB label \lmfdbclass{976}{a}, and contains a Jacobian with a
$29$-torsion point.  The remainder correspond to $23 = 9 + 9 + 4 + 1$ isogeny
classes consisting of exactly two abelian surfaces linked by isogenies of
degrees $11^4$, $13^4$, $17^2$, and~$31^2$ respectively, as listed in
\cref{table:degrees}.  The only class that took more than 1.5 hours is the
example discussed in \S \ref{sec:31isog}, featuring the isogeny of
degree~$31^2$.

The largest degree of an irreducible isogeny was $13^4$. For example, the class
\lmfdbclass{349}{a} has two abelian surfaces connected by such an isogeny,
namely the Jacobians of the two curves
\begin{align*}
  \lmfdbcurve{349}{a}{349}{1}  \colon& y^2 + (x^3 + x^2 + x + 1) y = -x^3 - x^2\\
  C \colon& y^2 + y = x^5 - 363 x^4 - 2517 x^3 + 151106 x^2 + 487525 x - 16355862.
\end{align*}

\Cref{table:graphs} continues our zoological study of isogeny graphs by listing
all graphs on at most four vertices that we observed. We label an edge
by~$\ell$ (resp.~$\ell^2$) when it corresponds to a 1-step (resp.~2-step)
$\ell$-isogeny.

\renewcommand{\arraystretch}{2}
\begin{table}[!ht]
\begin{tabular}{cl}
{\begin{tikzpicture}
    \node (0) at (0,0) {};
    \draw[fill] (0) circle (0.05);
    \node (1) at (2,0) {};
    \draw[fill] (1) circle (0.05);
    \draw (0) -- node[above] {$n$} (1);
\end{tikzpicture}}
&
{for $n\in \{2,3,5,7,13,17,31,9,25,49,121,169\}$}
\\

\begin{tikzpicture}
    \node (0) at (0,0) {};
    \draw[fill] (0) circle (0.05);
    \node (1) at (2,0) {};
    \draw[fill] (1) circle (0.05);
    \node (2) at (4,0) {};
    \draw[fill] (2) circle (0.05);
    \draw (0) -- node[above] {$n$} (1);
    \draw (1) -- node[above] {$n$} (2);
\end{tikzpicture}
& for $n\in \{3,5,7,9\}$ \\

\begin{tikzpicture}
    \node (1) at (90:1) {};
    \draw[fill] (1) circle (0.05);
    \node (2) at (210:1) {};
    \draw[fill] (2) circle (0.05);
    \node (3) at (330:1) {};
    \draw[fill] (3) circle (0.05);
    \draw (1) -- node[left] {$4$} (2);
    \draw (2) -- node[below] {$4$} (3);
    \draw (1) -- node[right] {$4$} (3);
\end{tikzpicture}
& \\

\begin{tikzpicture}
    \node (0) at (0,0) {};
    \draw[fill] (0) circle (0.05);
    \node (1) at (1.5,0) {};
    \draw[fill] (1) circle (0.05);
    \node (2) at (3,0) {};
    \draw[fill] (2) circle (0.05);
    \node (3) at (4.5,0) {};
    \draw[fill] (3) circle (0.05);
    \draw (0) -- node[above] {$n$} (1);
    \draw (1) -- node[above] {$n$} (2);
    \draw (2) -- node[above] {$n$} (3);
\end{tikzpicture}
& for $n\in \{3,5\}$ \\

{\begin{tikzpicture}
    \node (0) at (0,0) {};
    \draw[fill] (0) circle (0.05);
    \node (1) at (2,0) {};
    \draw[fill] (1) circle (0.05);
    \node (2) at (0,2) {};
    \draw[fill] (2) circle (0.05);
    \node (3) at (2,2) {};
    \draw[fill] (3) circle (0.05);
    \draw (0) -- node[above] {$n$} (1);
    \draw (2) -- node[above] {$n$} (3);
    \draw (0) -- node[left] {$m$} (2);
    \draw (1) -- node[right] {$m$} (3);
\end{tikzpicture}}
& \raisebox{2.5em}{
\renewcommand{\arraystretch}{1}
\begin{tabular}{@{}l@{}} \hspace{-1em} for $(m,n)\in \{(2,3), 2,5), (2,9), (2,25),$\\
\qquad\qquad\quad\  $(3,3), (5,9), (9,25)\}$\end{tabular}
\renewcommand{\arraystretch}{2}} \\

{\begin{tikzpicture}
    \node (0) at (0,0) {};
    \draw[fill] (0) circle (0.05);
    \node (1) at (90:1.5) {};
    \draw[fill] (1) circle (0.05);
    \node (2) at (210:1.5) {};
    \draw[fill] (2) circle (0.05);
    \node (3) at (330:1.5) {};
    \draw[fill] (3) circle (0.05);
    \draw (0) -- node[left] {$n$} (1);
    \draw (0) -- node[above] {$n$} (2);
    \draw (0) -- node[above] {$n$} (3);
\end{tikzpicture}
}
& \raisebox{3em}{for $n\in \{2, 3\}$}
\end{tabular}
\bigskip
\caption{Observed isogeny graphs with 2 to 4 vertices.}
\label{table:graphs}
\end{table}

The \numprint{1440894} isogeny classes represent 71 non-isomorphic graphs, the
largest having size 30 (see
\url{https://github.com/edgarcosta/genus2isogenies/tree/main/data/graphs_2e20}). In
every class containing 18 or more abelian surfaces, the only irreducible
isogenies are Richelot isogenies.

Despite the size of our dataset, this list
of graphs is not complete. The following curve, suggested by Noam Elkies, gives
rise to an isogeny graph consisting of 42 vertices connected by Richelot
isogenies:
$$y^2 = (x+4)(x+11)(4x-1)(12x+13)(15x-4).$$
This curve has conductor
$2^{24} \cdot 3^{3} \cdot 5^{2} \cdot 7^{2} \cdot 13 \cdot 17^{2}$ and
discriminant
$2^{40} \cdot 3^{18} \cdot 5^{6} \cdot 7^{6} \cdot 13^{4} \cdot 17^{4}$. Its
isogeny graph is displayed in \Cref{fig:rank42}.

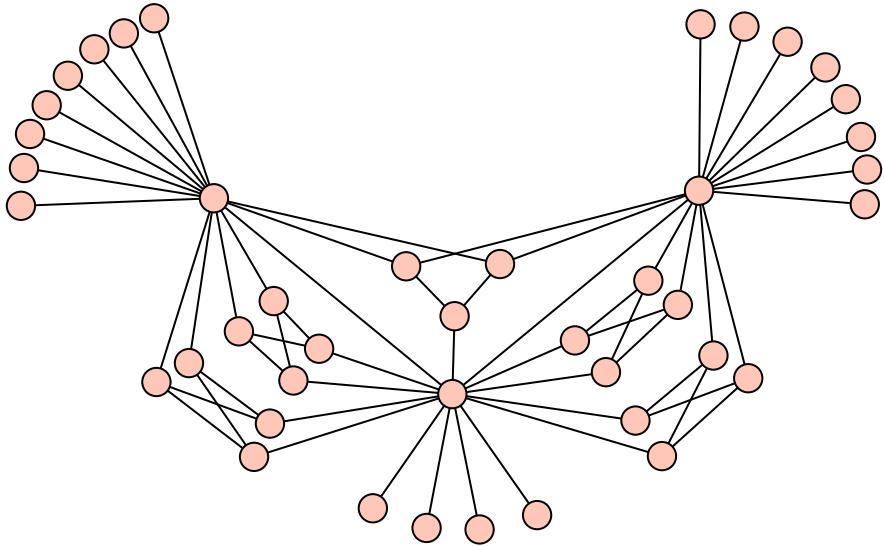
\begin{figure}[ht]
\centering
  \scalebox{0.755}{\input{rank42.pgf}}
  \caption{Isogeny graph with 42 vertices.}
  \label{fig:rank42}
\end{figure}

\subsection{Sanity checks}

Running such a large scale computation allows us to perform several sanity
checks regarding the correctness of our implementation.  For each isogeny
found, we heuristically confirmed that the Jacobians were isogenous in two
ways.  First, we confirmed that all the traces of Frobenius agree for all
primes up to $2^{16}$, using \cite{smalljac}, that do not divide the discriminants of the curves
involved. Second, independent analytic computations based on \cite{rigendos}
have confirmed the isogeny degrees that we computed.

Regarding completeness of the isogeny classes, we ran two checks. As indicated
above, we started from only one curve in each of the
\numprint{1440894} (heuristic) isogeny classes in our dataset, and we checked
that all the other curves in that class indeed appeared in our result. We also
confirmed that our isogeny classes are closed under Richelot isogenies, using
the function \texttt{RichelotIsogenousSurfaces} in \cite{Magma} based on
algebraic formulas specific to this case.

\printbibliography

\end{document}

%% file: rank42.pgf
\begingroup%
\makeatletter%
\begin{pgfpicture}%
\pgfpathrectangle{\pgfpointorigin}{\pgfqpoint{6.300000in}{3.888372in}}%
\pgfusepath{use as bounding box, clip}%
\begin{pgfscope}%
\pgfsetbuttcap%
\pgfsetmiterjoin%
\definecolor{currentfill}{rgb}{1.000000,1.000000,1.000000}%
\pgfsetfillcolor{currentfill}%
\pgfsetlinewidth{0.000000pt}%
\definecolor{currentstroke}{rgb}{1.000000,1.000000,1.000000}%
\pgfsetstrokecolor{currentstroke}%
\pgfsetdash{}{0pt}%
\pgfpathmoveto{\pgfqpoint{0.000000in}{0.000000in}}%
\pgfpathlineto{\pgfqpoint{6.300000in}{0.000000in}}%
\pgfpathlineto{\pgfqpoint{6.300000in}{3.888372in}}%
\pgfpathlineto{\pgfqpoint{0.000000in}{3.888372in}}%
\pgfpathlineto{\pgfqpoint{0.000000in}{0.000000in}}%
\pgfpathclose%
\pgfusepath{fill}%
\end{pgfscope}%
\begin{pgfscope}%
\pgfsetbuttcap%
\pgfsetmiterjoin%
\definecolor{currentfill}{rgb}{1.000000,1.000000,1.000000}%
\pgfsetfillcolor{currentfill}%
\pgfsetlinewidth{0.000000pt}%
\definecolor{currentstroke}{rgb}{0.000000,0.000000,0.000000}%
\pgfsetstrokecolor{currentstroke}%
\pgfsetstrokeopacity{0.000000}%
\pgfsetdash{}{0pt}%
\pgfpathmoveto{\pgfqpoint{0.100000in}{0.100000in}}%
\pgfpathlineto{\pgfqpoint{6.200000in}{0.100000in}}%
\pgfpathlineto{\pgfqpoint{6.200000in}{3.788372in}}%
\pgfpathlineto{\pgfqpoint{0.100000in}{3.788372in}}%
\pgfpathlineto{\pgfqpoint{0.100000in}{0.100000in}}%
\pgfpathclose%
\pgfusepath{fill}%
\end{pgfscope}%
\begin{pgfscope}%
\pgfpathrectangle{\pgfqpoint{0.100000in}{0.100000in}}{\pgfqpoint{6.100000in}{3.688372in}}%
\pgfusepath{clip}%
\pgfsetrectcap%
\pgfsetroundjoin%
\pgfsetlinewidth{1.003750pt}%
\definecolor{currentstroke}{rgb}{0.000000,0.000000,0.000000}%
\pgfsetstrokecolor{currentstroke}%
\pgfsetdash{}{0pt}%
\pgfpathmoveto{\pgfqpoint{6.082692in}{2.668178in}}%
\pgfpathlineto{\pgfqpoint{4.918009in}{2.521281in}}%
\pgfusepath{stroke}%
\end{pgfscope}%
\begin{pgfscope}%
\pgfpathrectangle{\pgfqpoint{0.100000in}{0.100000in}}{\pgfqpoint{6.100000in}{3.688372in}}%
\pgfusepath{clip}%
\pgfsetrectcap%
\pgfsetroundjoin%
\pgfsetlinewidth{1.003750pt}%
\definecolor{currentstroke}{rgb}{0.000000,0.000000,0.000000}%
\pgfsetstrokecolor{currentstroke}%
\pgfsetdash{}{0pt}%
\pgfpathmoveto{\pgfqpoint{3.396577in}{0.170930in}}%
\pgfpathlineto{\pgfqpoint{3.207710in}{1.110021in}}%
\pgfusepath{stroke}%
\end{pgfscope}%
\begin{pgfscope}%
\pgfpathrectangle{\pgfqpoint{0.100000in}{0.100000in}}{\pgfqpoint{6.100000in}{3.688372in}}%
\pgfusepath{clip}%
\pgfsetrectcap%
\pgfsetroundjoin%
\pgfsetlinewidth{1.003750pt}%
\definecolor{currentstroke}{rgb}{0.000000,0.000000,0.000000}%
\pgfsetstrokecolor{currentstroke}%
\pgfsetdash{}{0pt}%
\pgfpathmoveto{\pgfqpoint{5.232789in}{3.659732in}}%
\pgfpathlineto{\pgfqpoint{4.918009in}{2.521281in}}%
\pgfusepath{stroke}%
\end{pgfscope}%
\begin{pgfscope}%
\pgfpathrectangle{\pgfqpoint{0.100000in}{0.100000in}}{\pgfqpoint{6.100000in}{3.688372in}}%
\pgfusepath{clip}%
\pgfsetrectcap%
\pgfsetroundjoin%
\pgfsetlinewidth{1.003750pt}%
\definecolor{currentstroke}{rgb}{0.000000,0.000000,0.000000}%
\pgfsetstrokecolor{currentstroke}%
\pgfsetdash{}{0pt}%
\pgfpathmoveto{\pgfqpoint{5.531829in}{3.554806in}}%
\pgfpathlineto{\pgfqpoint{4.918009in}{2.521281in}}%
\pgfusepath{stroke}%
\end{pgfscope}%
\begin{pgfscope}%
\pgfpathrectangle{\pgfqpoint{0.100000in}{0.100000in}}{\pgfqpoint{6.100000in}{3.688372in}}%
\pgfusepath{clip}%
\pgfsetrectcap%
\pgfsetroundjoin%
\pgfsetlinewidth{1.003750pt}%
\definecolor{currentstroke}{rgb}{0.000000,0.000000,0.000000}%
\pgfsetstrokecolor{currentstroke}%
\pgfsetdash{}{0pt}%
\pgfpathmoveto{\pgfqpoint{3.223448in}{1.650392in}}%
\pgfpathlineto{\pgfqpoint{2.887684in}{1.996649in}}%
\pgfusepath{stroke}%
\end{pgfscope}%
\begin{pgfscope}%
\pgfpathrectangle{\pgfqpoint{0.100000in}{0.100000in}}{\pgfqpoint{6.100000in}{3.688372in}}%
\pgfusepath{clip}%
\pgfsetrectcap%
\pgfsetroundjoin%
\pgfsetlinewidth{1.003750pt}%
\definecolor{currentstroke}{rgb}{0.000000,0.000000,0.000000}%
\pgfsetstrokecolor{currentstroke}%
\pgfsetdash{}{0pt}%
\pgfpathmoveto{\pgfqpoint{3.223448in}{1.650392in}}%
\pgfpathlineto{\pgfqpoint{3.207710in}{1.110021in}}%
\pgfusepath{stroke}%
\end{pgfscope}%
\begin{pgfscope}%
\pgfpathrectangle{\pgfqpoint{0.100000in}{0.100000in}}{\pgfqpoint{6.100000in}{3.688372in}}%
\pgfusepath{clip}%
\pgfsetrectcap%
\pgfsetroundjoin%
\pgfsetlinewidth{1.003750pt}%
\definecolor{currentstroke}{rgb}{0.000000,0.000000,0.000000}%
\pgfsetstrokecolor{currentstroke}%
\pgfsetdash{}{0pt}%
\pgfpathmoveto{\pgfqpoint{3.223448in}{1.650392in}}%
\pgfpathlineto{\pgfqpoint{3.538228in}{2.012388in}}%
\pgfusepath{stroke}%
\end{pgfscope}%
\begin{pgfscope}%
\pgfpathrectangle{\pgfqpoint{0.100000in}{0.100000in}}{\pgfqpoint{6.100000in}{3.688372in}}%
\pgfusepath{clip}%
\pgfsetrectcap%
\pgfsetroundjoin%
\pgfsetlinewidth{1.003750pt}%
\definecolor{currentstroke}{rgb}{0.000000,0.000000,0.000000}%
\pgfsetstrokecolor{currentstroke}%
\pgfsetdash{}{0pt}%
\pgfpathmoveto{\pgfqpoint{4.477319in}{0.926400in}}%
\pgfpathlineto{\pgfqpoint{3.207710in}{1.110021in}}%
\pgfusepath{stroke}%
\end{pgfscope}%
\begin{pgfscope}%
\pgfpathrectangle{\pgfqpoint{0.100000in}{0.100000in}}{\pgfqpoint{6.100000in}{3.688372in}}%
\pgfusepath{clip}%
\pgfsetrectcap%
\pgfsetroundjoin%
\pgfsetlinewidth{1.003750pt}%
\definecolor{currentstroke}{rgb}{0.000000,0.000000,0.000000}%
\pgfsetstrokecolor{currentstroke}%
\pgfsetdash{}{0pt}%
\pgfpathmoveto{\pgfqpoint{4.477319in}{0.926400in}}%
\pgfpathlineto{\pgfqpoint{5.259020in}{1.220194in}}%
\pgfusepath{stroke}%
\end{pgfscope}%
\begin{pgfscope}%
\pgfpathrectangle{\pgfqpoint{0.100000in}{0.100000in}}{\pgfqpoint{6.100000in}{3.688372in}}%
\pgfusepath{clip}%
\pgfsetrectcap%
\pgfsetroundjoin%
\pgfsetlinewidth{1.003750pt}%
\definecolor{currentstroke}{rgb}{0.000000,0.000000,0.000000}%
\pgfsetstrokecolor{currentstroke}%
\pgfsetdash{}{0pt}%
\pgfpathmoveto{\pgfqpoint{4.477319in}{0.926400in}}%
\pgfpathlineto{\pgfqpoint{5.017690in}{1.377584in}}%
\pgfusepath{stroke}%
\end{pgfscope}%
\begin{pgfscope}%
\pgfpathrectangle{\pgfqpoint{0.100000in}{0.100000in}}{\pgfqpoint{6.100000in}{3.688372in}}%
\pgfusepath{clip}%
\pgfsetrectcap%
\pgfsetroundjoin%
\pgfsetlinewidth{1.003750pt}%
\definecolor{currentstroke}{rgb}{0.000000,0.000000,0.000000}%
\pgfsetstrokecolor{currentstroke}%
\pgfsetdash{}{0pt}%
\pgfpathmoveto{\pgfqpoint{1.156399in}{1.193962in}}%
\pgfpathlineto{\pgfqpoint{1.555119in}{2.468818in}}%
\pgfusepath{stroke}%
\end{pgfscope}%
\begin{pgfscope}%
\pgfpathrectangle{\pgfqpoint{0.100000in}{0.100000in}}{\pgfqpoint{6.100000in}{3.688372in}}%
\pgfusepath{clip}%
\pgfsetrectcap%
\pgfsetroundjoin%
\pgfsetlinewidth{1.003750pt}%
\definecolor{currentstroke}{rgb}{0.000000,0.000000,0.000000}%
\pgfsetstrokecolor{currentstroke}%
\pgfsetdash{}{0pt}%
\pgfpathmoveto{\pgfqpoint{1.156399in}{1.193962in}}%
\pgfpathlineto{\pgfqpoint{1.943347in}{0.905415in}}%
\pgfusepath{stroke}%
\end{pgfscope}%
\begin{pgfscope}%
\pgfpathrectangle{\pgfqpoint{0.100000in}{0.100000in}}{\pgfqpoint{6.100000in}{3.688372in}}%
\pgfusepath{clip}%
\pgfsetrectcap%
\pgfsetroundjoin%
\pgfsetlinewidth{1.003750pt}%
\definecolor{currentstroke}{rgb}{0.000000,0.000000,0.000000}%
\pgfsetstrokecolor{currentstroke}%
\pgfsetdash{}{0pt}%
\pgfpathmoveto{\pgfqpoint{1.156399in}{1.193962in}}%
\pgfpathlineto{\pgfqpoint{1.833174in}{0.674577in}}%
\pgfusepath{stroke}%
\end{pgfscope}%
\begin{pgfscope}%
\pgfpathrectangle{\pgfqpoint{0.100000in}{0.100000in}}{\pgfqpoint{6.100000in}{3.688372in}}%
\pgfusepath{clip}%
\pgfsetrectcap%
\pgfsetroundjoin%
\pgfsetlinewidth{1.003750pt}%
\definecolor{currentstroke}{rgb}{0.000000,0.000000,0.000000}%
\pgfsetstrokecolor{currentstroke}%
\pgfsetdash{}{0pt}%
\pgfpathmoveto{\pgfqpoint{2.887684in}{1.996649in}}%
\pgfpathlineto{\pgfqpoint{4.918009in}{2.521281in}}%
\pgfusepath{stroke}%
\end{pgfscope}%
\begin{pgfscope}%
\pgfpathrectangle{\pgfqpoint{0.100000in}{0.100000in}}{\pgfqpoint{6.100000in}{3.688372in}}%
\pgfusepath{clip}%
\pgfsetrectcap%
\pgfsetroundjoin%
\pgfsetlinewidth{1.003750pt}%
\definecolor{currentstroke}{rgb}{0.000000,0.000000,0.000000}%
\pgfsetstrokecolor{currentstroke}%
\pgfsetdash{}{0pt}%
\pgfpathmoveto{\pgfqpoint{2.887684in}{1.996649in}}%
\pgfpathlineto{\pgfqpoint{1.555119in}{2.468818in}}%
\pgfusepath{stroke}%
\end{pgfscope}%
\begin{pgfscope}%
\pgfpathrectangle{\pgfqpoint{0.100000in}{0.100000in}}{\pgfqpoint{6.100000in}{3.688372in}}%
\pgfusepath{clip}%
\pgfsetrectcap%
\pgfsetroundjoin%
\pgfsetlinewidth{1.003750pt}%
\definecolor{currentstroke}{rgb}{0.000000,0.000000,0.000000}%
\pgfsetstrokecolor{currentstroke}%
\pgfsetdash{}{0pt}%
\pgfpathmoveto{\pgfqpoint{0.238293in}{2.678671in}}%
\pgfpathlineto{\pgfqpoint{1.555119in}{2.468818in}}%
\pgfusepath{stroke}%
\end{pgfscope}%
\begin{pgfscope}%
\pgfpathrectangle{\pgfqpoint{0.100000in}{0.100000in}}{\pgfqpoint{6.100000in}{3.688372in}}%
\pgfusepath{clip}%
\pgfsetrectcap%
\pgfsetroundjoin%
\pgfsetlinewidth{1.003750pt}%
\definecolor{currentstroke}{rgb}{0.000000,0.000000,0.000000}%
\pgfsetstrokecolor{currentstroke}%
\pgfsetdash{}{0pt}%
\pgfpathmoveto{\pgfqpoint{0.542579in}{3.318722in}}%
\pgfpathlineto{\pgfqpoint{1.555119in}{2.468818in}}%
\pgfusepath{stroke}%
\end{pgfscope}%
\begin{pgfscope}%
\pgfpathrectangle{\pgfqpoint{0.100000in}{0.100000in}}{\pgfqpoint{6.100000in}{3.688372in}}%
\pgfusepath{clip}%
\pgfsetrectcap%
\pgfsetroundjoin%
\pgfsetlinewidth{1.003750pt}%
\definecolor{currentstroke}{rgb}{0.000000,0.000000,0.000000}%
\pgfsetstrokecolor{currentstroke}%
\pgfsetdash{}{0pt}%
\pgfpathmoveto{\pgfqpoint{5.017690in}{1.377584in}}%
\pgfpathlineto{\pgfqpoint{4.918009in}{2.521281in}}%
\pgfusepath{stroke}%
\end{pgfscope}%
\begin{pgfscope}%
\pgfpathrectangle{\pgfqpoint{0.100000in}{0.100000in}}{\pgfqpoint{6.100000in}{3.688372in}}%
\pgfusepath{clip}%
\pgfsetrectcap%
\pgfsetroundjoin%
\pgfsetlinewidth{1.003750pt}%
\definecolor{currentstroke}{rgb}{0.000000,0.000000,0.000000}%
\pgfsetstrokecolor{currentstroke}%
\pgfsetdash{}{0pt}%
\pgfpathmoveto{\pgfqpoint{5.017690in}{1.377584in}}%
\pgfpathlineto{\pgfqpoint{4.660940in}{0.679823in}}%
\pgfusepath{stroke}%
\end{pgfscope}%
\begin{pgfscope}%
\pgfpathrectangle{\pgfqpoint{0.100000in}{0.100000in}}{\pgfqpoint{6.100000in}{3.688372in}}%
\pgfusepath{clip}%
\pgfsetrectcap%
\pgfsetroundjoin%
\pgfsetlinewidth{1.003750pt}%
\definecolor{currentstroke}{rgb}{0.000000,0.000000,0.000000}%
\pgfsetstrokecolor{currentstroke}%
\pgfsetdash{}{0pt}%
\pgfpathmoveto{\pgfqpoint{2.284357in}{1.424800in}}%
\pgfpathlineto{\pgfqpoint{1.969578in}{1.755319in}}%
\pgfusepath{stroke}%
\end{pgfscope}%
\begin{pgfscope}%
\pgfpathrectangle{\pgfqpoint{0.100000in}{0.100000in}}{\pgfqpoint{6.100000in}{3.688372in}}%
\pgfusepath{clip}%
\pgfsetrectcap%
\pgfsetroundjoin%
\pgfsetlinewidth{1.003750pt}%
\definecolor{currentstroke}{rgb}{0.000000,0.000000,0.000000}%
\pgfsetstrokecolor{currentstroke}%
\pgfsetdash{}{0pt}%
\pgfpathmoveto{\pgfqpoint{2.284357in}{1.424800in}}%
\pgfpathlineto{\pgfqpoint{3.207710in}{1.110021in}}%
\pgfusepath{stroke}%
\end{pgfscope}%
\begin{pgfscope}%
\pgfpathrectangle{\pgfqpoint{0.100000in}{0.100000in}}{\pgfqpoint{6.100000in}{3.688372in}}%
\pgfusepath{clip}%
\pgfsetrectcap%
\pgfsetroundjoin%
\pgfsetlinewidth{1.003750pt}%
\definecolor{currentstroke}{rgb}{0.000000,0.000000,0.000000}%
\pgfsetstrokecolor{currentstroke}%
\pgfsetdash{}{0pt}%
\pgfpathmoveto{\pgfqpoint{2.284357in}{1.424800in}}%
\pgfpathlineto{\pgfqpoint{1.728248in}{1.545466in}}%
\pgfusepath{stroke}%
\end{pgfscope}%
\begin{pgfscope}%
\pgfpathrectangle{\pgfqpoint{0.100000in}{0.100000in}}{\pgfqpoint{6.100000in}{3.688372in}}%
\pgfusepath{clip}%
\pgfsetrectcap%
\pgfsetroundjoin%
\pgfsetlinewidth{1.003750pt}%
\definecolor{currentstroke}{rgb}{0.000000,0.000000,0.000000}%
\pgfsetstrokecolor{currentstroke}%
\pgfsetdash{}{0pt}%
\pgfpathmoveto{\pgfqpoint{3.029335in}{0.181423in}}%
\pgfpathlineto{\pgfqpoint{3.207710in}{1.110021in}}%
\pgfusepath{stroke}%
\end{pgfscope}%
\begin{pgfscope}%
\pgfpathrectangle{\pgfqpoint{0.100000in}{0.100000in}}{\pgfqpoint{6.100000in}{3.688372in}}%
\pgfusepath{clip}%
\pgfsetrectcap%
\pgfsetroundjoin%
\pgfsetlinewidth{1.003750pt}%
\definecolor{currentstroke}{rgb}{0.000000,0.000000,0.000000}%
\pgfsetstrokecolor{currentstroke}%
\pgfsetdash{}{0pt}%
\pgfpathmoveto{\pgfqpoint{1.140660in}{3.717442in}}%
\pgfpathlineto{\pgfqpoint{1.555119in}{2.468818in}}%
\pgfusepath{stroke}%
\end{pgfscope}%
\begin{pgfscope}%
\pgfpathrectangle{\pgfqpoint{0.100000in}{0.100000in}}{\pgfqpoint{6.100000in}{3.688372in}}%
\pgfusepath{clip}%
\pgfsetrectcap%
\pgfsetroundjoin%
\pgfsetlinewidth{1.003750pt}%
\definecolor{currentstroke}{rgb}{0.000000,0.000000,0.000000}%
\pgfsetstrokecolor{currentstroke}%
\pgfsetdash{}{0pt}%
\pgfpathmoveto{\pgfqpoint{1.728248in}{1.545466in}}%
\pgfpathlineto{\pgfqpoint{2.105983in}{1.204455in}}%
\pgfusepath{stroke}%
\end{pgfscope}%
\begin{pgfscope}%
\pgfpathrectangle{\pgfqpoint{0.100000in}{0.100000in}}{\pgfqpoint{6.100000in}{3.688372in}}%
\pgfusepath{clip}%
\pgfsetrectcap%
\pgfsetroundjoin%
\pgfsetlinewidth{1.003750pt}%
\definecolor{currentstroke}{rgb}{0.000000,0.000000,0.000000}%
\pgfsetstrokecolor{currentstroke}%
\pgfsetdash{}{0pt}%
\pgfpathmoveto{\pgfqpoint{1.728248in}{1.545466in}}%
\pgfpathlineto{\pgfqpoint{1.555119in}{2.468818in}}%
\pgfusepath{stroke}%
\end{pgfscope}%
\begin{pgfscope}%
\pgfpathrectangle{\pgfqpoint{0.100000in}{0.100000in}}{\pgfqpoint{6.100000in}{3.688372in}}%
\pgfusepath{clip}%
\pgfsetrectcap%
\pgfsetroundjoin%
\pgfsetlinewidth{1.003750pt}%
\definecolor{currentstroke}{rgb}{0.000000,0.000000,0.000000}%
\pgfsetstrokecolor{currentstroke}%
\pgfsetdash{}{0pt}%
\pgfpathmoveto{\pgfqpoint{0.395683in}{3.114115in}}%
\pgfpathlineto{\pgfqpoint{1.555119in}{2.468818in}}%
\pgfusepath{stroke}%
\end{pgfscope}%
\begin{pgfscope}%
\pgfpathrectangle{\pgfqpoint{0.100000in}{0.100000in}}{\pgfqpoint{6.100000in}{3.688372in}}%
\pgfusepath{clip}%
\pgfsetrectcap%
\pgfsetroundjoin%
\pgfsetlinewidth{1.003750pt}%
\definecolor{currentstroke}{rgb}{0.000000,0.000000,0.000000}%
\pgfsetstrokecolor{currentstroke}%
\pgfsetdash{}{0pt}%
\pgfpathmoveto{\pgfqpoint{2.105983in}{1.204455in}}%
\pgfpathlineto{\pgfqpoint{1.969578in}{1.755319in}}%
\pgfusepath{stroke}%
\end{pgfscope}%
\begin{pgfscope}%
\pgfpathrectangle{\pgfqpoint{0.100000in}{0.100000in}}{\pgfqpoint{6.100000in}{3.688372in}}%
\pgfusepath{clip}%
\pgfsetrectcap%
\pgfsetroundjoin%
\pgfsetlinewidth{1.003750pt}%
\definecolor{currentstroke}{rgb}{0.000000,0.000000,0.000000}%
\pgfsetstrokecolor{currentstroke}%
\pgfsetdash{}{0pt}%
\pgfpathmoveto{\pgfqpoint{2.105983in}{1.204455in}}%
\pgfpathlineto{\pgfqpoint{3.207710in}{1.110021in}}%
\pgfusepath{stroke}%
\end{pgfscope}%
\begin{pgfscope}%
\pgfpathrectangle{\pgfqpoint{0.100000in}{0.100000in}}{\pgfqpoint{6.100000in}{3.688372in}}%
\pgfusepath{clip}%
\pgfsetrectcap%
\pgfsetroundjoin%
\pgfsetlinewidth{1.003750pt}%
\definecolor{currentstroke}{rgb}{0.000000,0.000000,0.000000}%
\pgfsetstrokecolor{currentstroke}%
\pgfsetdash{}{0pt}%
\pgfpathmoveto{\pgfqpoint{0.930807in}{3.612515in}}%
\pgfpathlineto{\pgfqpoint{1.555119in}{2.468818in}}%
\pgfusepath{stroke}%
\end{pgfscope}%
\begin{pgfscope}%
\pgfpathrectangle{\pgfqpoint{0.100000in}{0.100000in}}{\pgfqpoint{6.100000in}{3.688372in}}%
\pgfusepath{clip}%
\pgfsetrectcap%
\pgfsetroundjoin%
\pgfsetlinewidth{1.003750pt}%
\definecolor{currentstroke}{rgb}{0.000000,0.000000,0.000000}%
\pgfsetstrokecolor{currentstroke}%
\pgfsetdash{}{0pt}%
\pgfpathmoveto{\pgfqpoint{4.660940in}{0.679823in}}%
\pgfpathlineto{\pgfqpoint{3.207710in}{1.110021in}}%
\pgfusepath{stroke}%
\end{pgfscope}%
\begin{pgfscope}%
\pgfpathrectangle{\pgfqpoint{0.100000in}{0.100000in}}{\pgfqpoint{6.100000in}{3.688372in}}%
\pgfusepath{clip}%
\pgfsetrectcap%
\pgfsetroundjoin%
\pgfsetlinewidth{1.003750pt}%
\definecolor{currentstroke}{rgb}{0.000000,0.000000,0.000000}%
\pgfsetstrokecolor{currentstroke}%
\pgfsetdash{}{0pt}%
\pgfpathmoveto{\pgfqpoint{4.660940in}{0.679823in}}%
\pgfpathlineto{\pgfqpoint{5.259020in}{1.220194in}}%
\pgfusepath{stroke}%
\end{pgfscope}%
\begin{pgfscope}%
\pgfpathrectangle{\pgfqpoint{0.100000in}{0.100000in}}{\pgfqpoint{6.100000in}{3.688372in}}%
\pgfusepath{clip}%
\pgfsetrectcap%
\pgfsetroundjoin%
\pgfsetlinewidth{1.003750pt}%
\definecolor{currentstroke}{rgb}{0.000000,0.000000,0.000000}%
\pgfsetstrokecolor{currentstroke}%
\pgfsetdash{}{0pt}%
\pgfpathmoveto{\pgfqpoint{0.280264in}{2.914755in}}%
\pgfpathlineto{\pgfqpoint{1.555119in}{2.468818in}}%
\pgfusepath{stroke}%
\end{pgfscope}%
\begin{pgfscope}%
\pgfpathrectangle{\pgfqpoint{0.100000in}{0.100000in}}{\pgfqpoint{6.100000in}{3.688372in}}%
\pgfusepath{clip}%
\pgfsetrectcap%
\pgfsetroundjoin%
\pgfsetlinewidth{1.003750pt}%
\definecolor{currentstroke}{rgb}{0.000000,0.000000,0.000000}%
\pgfsetstrokecolor{currentstroke}%
\pgfsetdash{}{0pt}%
\pgfpathmoveto{\pgfqpoint{1.555119in}{2.468818in}}%
\pgfpathlineto{\pgfqpoint{0.726201in}{3.502343in}}%
\pgfusepath{stroke}%
\end{pgfscope}%
\begin{pgfscope}%
\pgfpathrectangle{\pgfqpoint{0.100000in}{0.100000in}}{\pgfqpoint{6.100000in}{3.688372in}}%
\pgfusepath{clip}%
\pgfsetrectcap%
\pgfsetroundjoin%
\pgfsetlinewidth{1.003750pt}%
\definecolor{currentstroke}{rgb}{0.000000,0.000000,0.000000}%
\pgfsetstrokecolor{currentstroke}%
\pgfsetdash{}{0pt}%
\pgfpathmoveto{\pgfqpoint{1.555119in}{2.468818in}}%
\pgfpathlineto{\pgfqpoint{1.381991in}{1.325120in}}%
\pgfusepath{stroke}%
\end{pgfscope}%
\begin{pgfscope}%
\pgfpathrectangle{\pgfqpoint{0.100000in}{0.100000in}}{\pgfqpoint{6.100000in}{3.688372in}}%
\pgfusepath{clip}%
\pgfsetrectcap%
\pgfsetroundjoin%
\pgfsetlinewidth{1.003750pt}%
\definecolor{currentstroke}{rgb}{0.000000,0.000000,0.000000}%
\pgfsetstrokecolor{currentstroke}%
\pgfsetdash{}{0pt}%
\pgfpathmoveto{\pgfqpoint{1.555119in}{2.468818in}}%
\pgfpathlineto{\pgfqpoint{1.969578in}{1.755319in}}%
\pgfusepath{stroke}%
\end{pgfscope}%
\begin{pgfscope}%
\pgfpathrectangle{\pgfqpoint{0.100000in}{0.100000in}}{\pgfqpoint{6.100000in}{3.688372in}}%
\pgfusepath{clip}%
\pgfsetrectcap%
\pgfsetroundjoin%
\pgfsetlinewidth{1.003750pt}%
\definecolor{currentstroke}{rgb}{0.000000,0.000000,0.000000}%
\pgfsetstrokecolor{currentstroke}%
\pgfsetdash{}{0pt}%
\pgfpathmoveto{\pgfqpoint{1.555119in}{2.468818in}}%
\pgfpathlineto{\pgfqpoint{3.207710in}{1.110021in}}%
\pgfusepath{stroke}%
\end{pgfscope}%
\begin{pgfscope}%
\pgfpathrectangle{\pgfqpoint{0.100000in}{0.100000in}}{\pgfqpoint{6.100000in}{3.688372in}}%
\pgfusepath{clip}%
\pgfsetrectcap%
\pgfsetroundjoin%
\pgfsetlinewidth{1.003750pt}%
\definecolor{currentstroke}{rgb}{0.000000,0.000000,0.000000}%
\pgfsetstrokecolor{currentstroke}%
\pgfsetdash{}{0pt}%
\pgfpathmoveto{\pgfqpoint{1.555119in}{2.468818in}}%
\pgfpathlineto{\pgfqpoint{3.538228in}{2.012388in}}%
\pgfusepath{stroke}%
\end{pgfscope}%
\begin{pgfscope}%
\pgfpathrectangle{\pgfqpoint{0.100000in}{0.100000in}}{\pgfqpoint{6.100000in}{3.688372in}}%
\pgfusepath{clip}%
\pgfsetrectcap%
\pgfsetroundjoin%
\pgfsetlinewidth{1.003750pt}%
\definecolor{currentstroke}{rgb}{0.000000,0.000000,0.000000}%
\pgfsetstrokecolor{currentstroke}%
\pgfsetdash{}{0pt}%
\pgfpathmoveto{\pgfqpoint{1.555119in}{2.468818in}}%
\pgfpathlineto{\pgfqpoint{0.217308in}{2.416355in}}%
\pgfusepath{stroke}%
\end{pgfscope}%
\begin{pgfscope}%
\pgfpathrectangle{\pgfqpoint{0.100000in}{0.100000in}}{\pgfqpoint{6.100000in}{3.688372in}}%
\pgfusepath{clip}%
\pgfsetrectcap%
\pgfsetroundjoin%
\pgfsetlinewidth{1.003750pt}%
\definecolor{currentstroke}{rgb}{0.000000,0.000000,0.000000}%
\pgfsetstrokecolor{currentstroke}%
\pgfsetdash{}{0pt}%
\pgfpathmoveto{\pgfqpoint{2.656846in}{0.317827in}}%
\pgfpathlineto{\pgfqpoint{3.207710in}{1.110021in}}%
\pgfusepath{stroke}%
\end{pgfscope}%
\begin{pgfscope}%
\pgfpathrectangle{\pgfqpoint{0.100000in}{0.100000in}}{\pgfqpoint{6.100000in}{3.688372in}}%
\pgfusepath{clip}%
\pgfsetrectcap%
\pgfsetroundjoin%
\pgfsetlinewidth{1.003750pt}%
\definecolor{currentstroke}{rgb}{0.000000,0.000000,0.000000}%
\pgfsetstrokecolor{currentstroke}%
\pgfsetdash{}{0pt}%
\pgfpathmoveto{\pgfqpoint{3.207710in}{1.110021in}}%
\pgfpathlineto{\pgfqpoint{4.918009in}{2.521281in}}%
\pgfusepath{stroke}%
\end{pgfscope}%
\begin{pgfscope}%
\pgfpathrectangle{\pgfqpoint{0.100000in}{0.100000in}}{\pgfqpoint{6.100000in}{3.688372in}}%
\pgfusepath{clip}%
\pgfsetrectcap%
\pgfsetroundjoin%
\pgfsetlinewidth{1.003750pt}%
\definecolor{currentstroke}{rgb}{0.000000,0.000000,0.000000}%
\pgfsetstrokecolor{currentstroke}%
\pgfsetdash{}{0pt}%
\pgfpathmoveto{\pgfqpoint{3.207710in}{1.110021in}}%
\pgfpathlineto{\pgfqpoint{1.943347in}{0.905415in}}%
\pgfusepath{stroke}%
\end{pgfscope}%
\begin{pgfscope}%
\pgfpathrectangle{\pgfqpoint{0.100000in}{0.100000in}}{\pgfqpoint{6.100000in}{3.688372in}}%
\pgfusepath{clip}%
\pgfsetrectcap%
\pgfsetroundjoin%
\pgfsetlinewidth{1.003750pt}%
\definecolor{currentstroke}{rgb}{0.000000,0.000000,0.000000}%
\pgfsetstrokecolor{currentstroke}%
\pgfsetdash{}{0pt}%
\pgfpathmoveto{\pgfqpoint{3.207710in}{1.110021in}}%
\pgfpathlineto{\pgfqpoint{1.833174in}{0.674577in}}%
\pgfusepath{stroke}%
\end{pgfscope}%
\begin{pgfscope}%
\pgfpathrectangle{\pgfqpoint{0.100000in}{0.100000in}}{\pgfqpoint{6.100000in}{3.688372in}}%
\pgfusepath{clip}%
\pgfsetrectcap%
\pgfsetroundjoin%
\pgfsetlinewidth{1.003750pt}%
\definecolor{currentstroke}{rgb}{0.000000,0.000000,0.000000}%
\pgfsetstrokecolor{currentstroke}%
\pgfsetdash{}{0pt}%
\pgfpathmoveto{\pgfqpoint{3.207710in}{1.110021in}}%
\pgfpathlineto{\pgfqpoint{3.795297in}{0.270610in}}%
\pgfusepath{stroke}%
\end{pgfscope}%
\begin{pgfscope}%
\pgfpathrectangle{\pgfqpoint{0.100000in}{0.100000in}}{\pgfqpoint{6.100000in}{3.688372in}}%
\pgfusepath{clip}%
\pgfsetrectcap%
\pgfsetroundjoin%
\pgfsetlinewidth{1.003750pt}%
\definecolor{currentstroke}{rgb}{0.000000,0.000000,0.000000}%
\pgfsetstrokecolor{currentstroke}%
\pgfsetdash{}{0pt}%
\pgfpathmoveto{\pgfqpoint{3.207710in}{1.110021in}}%
\pgfpathlineto{\pgfqpoint{4.057613in}{1.482510in}}%
\pgfusepath{stroke}%
\end{pgfscope}%
\begin{pgfscope}%
\pgfpathrectangle{\pgfqpoint{0.100000in}{0.100000in}}{\pgfqpoint{6.100000in}{3.688372in}}%
\pgfusepath{clip}%
\pgfsetrectcap%
\pgfsetroundjoin%
\pgfsetlinewidth{1.003750pt}%
\definecolor{currentstroke}{rgb}{0.000000,0.000000,0.000000}%
\pgfsetstrokecolor{currentstroke}%
\pgfsetdash{}{0pt}%
\pgfpathmoveto{\pgfqpoint{3.207710in}{1.110021in}}%
\pgfpathlineto{\pgfqpoint{4.272712in}{1.262165in}}%
\pgfusepath{stroke}%
\end{pgfscope}%
\begin{pgfscope}%
\pgfpathrectangle{\pgfqpoint{0.100000in}{0.100000in}}{\pgfqpoint{6.100000in}{3.688372in}}%
\pgfusepath{clip}%
\pgfsetrectcap%
\pgfsetroundjoin%
\pgfsetlinewidth{1.003750pt}%
\definecolor{currentstroke}{rgb}{0.000000,0.000000,0.000000}%
\pgfsetstrokecolor{currentstroke}%
\pgfsetdash{}{0pt}%
\pgfpathmoveto{\pgfqpoint{5.259020in}{1.220194in}}%
\pgfpathlineto{\pgfqpoint{4.918009in}{2.521281in}}%
\pgfusepath{stroke}%
\end{pgfscope}%
\begin{pgfscope}%
\pgfpathrectangle{\pgfqpoint{0.100000in}{0.100000in}}{\pgfqpoint{6.100000in}{3.688372in}}%
\pgfusepath{clip}%
\pgfsetrectcap%
\pgfsetroundjoin%
\pgfsetlinewidth{1.003750pt}%
\definecolor{currentstroke}{rgb}{0.000000,0.000000,0.000000}%
\pgfsetstrokecolor{currentstroke}%
\pgfsetdash{}{0pt}%
\pgfpathmoveto{\pgfqpoint{3.538228in}{2.012388in}}%
\pgfpathlineto{\pgfqpoint{4.918009in}{2.521281in}}%
\pgfusepath{stroke}%
\end{pgfscope}%
\begin{pgfscope}%
\pgfpathrectangle{\pgfqpoint{0.100000in}{0.100000in}}{\pgfqpoint{6.100000in}{3.688372in}}%
\pgfusepath{clip}%
\pgfsetrectcap%
\pgfsetroundjoin%
\pgfsetlinewidth{1.003750pt}%
\definecolor{currentstroke}{rgb}{0.000000,0.000000,0.000000}%
\pgfsetstrokecolor{currentstroke}%
\pgfsetdash{}{0pt}%
\pgfpathmoveto{\pgfqpoint{1.833174in}{0.674577in}}%
\pgfpathlineto{\pgfqpoint{1.381991in}{1.325120in}}%
\pgfusepath{stroke}%
\end{pgfscope}%
\begin{pgfscope}%
\pgfpathrectangle{\pgfqpoint{0.100000in}{0.100000in}}{\pgfqpoint{6.100000in}{3.688372in}}%
\pgfusepath{clip}%
\pgfsetrectcap%
\pgfsetroundjoin%
\pgfsetlinewidth{1.003750pt}%
\definecolor{currentstroke}{rgb}{0.000000,0.000000,0.000000}%
\pgfsetstrokecolor{currentstroke}%
\pgfsetdash{}{0pt}%
\pgfpathmoveto{\pgfqpoint{4.566506in}{1.896969in}}%
\pgfpathlineto{\pgfqpoint{4.918009in}{2.521281in}}%
\pgfusepath{stroke}%
\end{pgfscope}%
\begin{pgfscope}%
\pgfpathrectangle{\pgfqpoint{0.100000in}{0.100000in}}{\pgfqpoint{6.100000in}{3.688372in}}%
\pgfusepath{clip}%
\pgfsetrectcap%
\pgfsetroundjoin%
\pgfsetlinewidth{1.003750pt}%
\definecolor{currentstroke}{rgb}{0.000000,0.000000,0.000000}%
\pgfsetstrokecolor{currentstroke}%
\pgfsetdash{}{0pt}%
\pgfpathmoveto{\pgfqpoint{4.566506in}{1.896969in}}%
\pgfpathlineto{\pgfqpoint{4.057613in}{1.482510in}}%
\pgfusepath{stroke}%
\end{pgfscope}%
\begin{pgfscope}%
\pgfpathrectangle{\pgfqpoint{0.100000in}{0.100000in}}{\pgfqpoint{6.100000in}{3.688372in}}%
\pgfusepath{clip}%
\pgfsetrectcap%
\pgfsetroundjoin%
\pgfsetlinewidth{1.003750pt}%
\definecolor{currentstroke}{rgb}{0.000000,0.000000,0.000000}%
\pgfsetstrokecolor{currentstroke}%
\pgfsetdash{}{0pt}%
\pgfpathmoveto{\pgfqpoint{4.566506in}{1.896969in}}%
\pgfpathlineto{\pgfqpoint{4.272712in}{1.262165in}}%
\pgfusepath{stroke}%
\end{pgfscope}%
\begin{pgfscope}%
\pgfpathrectangle{\pgfqpoint{0.100000in}{0.100000in}}{\pgfqpoint{6.100000in}{3.688372in}}%
\pgfusepath{clip}%
\pgfsetrectcap%
\pgfsetroundjoin%
\pgfsetlinewidth{1.003750pt}%
\definecolor{currentstroke}{rgb}{0.000000,0.000000,0.000000}%
\pgfsetstrokecolor{currentstroke}%
\pgfsetdash{}{0pt}%
\pgfpathmoveto{\pgfqpoint{4.057613in}{1.482510in}}%
\pgfpathlineto{\pgfqpoint{4.771113in}{1.729087in}}%
\pgfusepath{stroke}%
\end{pgfscope}%
\begin{pgfscope}%
\pgfpathrectangle{\pgfqpoint{0.100000in}{0.100000in}}{\pgfqpoint{6.100000in}{3.688372in}}%
\pgfusepath{clip}%
\pgfsetrectcap%
\pgfsetroundjoin%
\pgfsetlinewidth{1.003750pt}%
\definecolor{currentstroke}{rgb}{0.000000,0.000000,0.000000}%
\pgfsetstrokecolor{currentstroke}%
\pgfsetdash{}{0pt}%
\pgfpathmoveto{\pgfqpoint{4.272712in}{1.262165in}}%
\pgfpathlineto{\pgfqpoint{4.771113in}{1.729087in}}%
\pgfusepath{stroke}%
\end{pgfscope}%
\begin{pgfscope}%
\pgfpathrectangle{\pgfqpoint{0.100000in}{0.100000in}}{\pgfqpoint{6.100000in}{3.688372in}}%
\pgfusepath{clip}%
\pgfsetrectcap%
\pgfsetroundjoin%
\pgfsetlinewidth{1.003750pt}%
\definecolor{currentstroke}{rgb}{0.000000,0.000000,0.000000}%
\pgfsetstrokecolor{currentstroke}%
\pgfsetdash{}{0pt}%
\pgfpathmoveto{\pgfqpoint{4.918009in}{2.521281in}}%
\pgfpathlineto{\pgfqpoint{5.794145in}{3.376431in}}%
\pgfusepath{stroke}%
\end{pgfscope}%
\begin{pgfscope}%
\pgfpathrectangle{\pgfqpoint{0.100000in}{0.100000in}}{\pgfqpoint{6.100000in}{3.688372in}}%
\pgfusepath{clip}%
\pgfsetrectcap%
\pgfsetroundjoin%
\pgfsetlinewidth{1.003750pt}%
\definecolor{currentstroke}{rgb}{0.000000,0.000000,0.000000}%
\pgfsetstrokecolor{currentstroke}%
\pgfsetdash{}{0pt}%
\pgfpathmoveto{\pgfqpoint{4.918009in}{2.521281in}}%
\pgfpathlineto{\pgfqpoint{6.066953in}{2.426847in}}%
\pgfusepath{stroke}%
\end{pgfscope}%
\begin{pgfscope}%
\pgfpathrectangle{\pgfqpoint{0.100000in}{0.100000in}}{\pgfqpoint{6.100000in}{3.688372in}}%
\pgfusepath{clip}%
\pgfsetrectcap%
\pgfsetroundjoin%
\pgfsetlinewidth{1.003750pt}%
\definecolor{currentstroke}{rgb}{0.000000,0.000000,0.000000}%
\pgfsetstrokecolor{currentstroke}%
\pgfsetdash{}{0pt}%
\pgfpathmoveto{\pgfqpoint{4.918009in}{2.521281in}}%
\pgfpathlineto{\pgfqpoint{5.935795in}{3.156086in}}%
\pgfusepath{stroke}%
\end{pgfscope}%
\begin{pgfscope}%
\pgfpathrectangle{\pgfqpoint{0.100000in}{0.100000in}}{\pgfqpoint{6.100000in}{3.688372in}}%
\pgfusepath{clip}%
\pgfsetrectcap%
\pgfsetroundjoin%
\pgfsetlinewidth{1.003750pt}%
\definecolor{currentstroke}{rgb}{0.000000,0.000000,0.000000}%
\pgfsetstrokecolor{currentstroke}%
\pgfsetdash{}{0pt}%
\pgfpathmoveto{\pgfqpoint{4.918009in}{2.521281in}}%
\pgfpathlineto{\pgfqpoint{4.771113in}{1.729087in}}%
\pgfusepath{stroke}%
\end{pgfscope}%
\begin{pgfscope}%
\pgfpathrectangle{\pgfqpoint{0.100000in}{0.100000in}}{\pgfqpoint{6.100000in}{3.688372in}}%
\pgfusepath{clip}%
\pgfsetrectcap%
\pgfsetroundjoin%
\pgfsetlinewidth{1.003750pt}%
\definecolor{currentstroke}{rgb}{0.000000,0.000000,0.000000}%
\pgfsetstrokecolor{currentstroke}%
\pgfsetdash{}{0pt}%
\pgfpathmoveto{\pgfqpoint{4.918009in}{2.521281in}}%
\pgfpathlineto{\pgfqpoint{4.928502in}{3.675471in}}%
\pgfusepath{stroke}%
\end{pgfscope}%
\begin{pgfscope}%
\pgfpathrectangle{\pgfqpoint{0.100000in}{0.100000in}}{\pgfqpoint{6.100000in}{3.688372in}}%
\pgfusepath{clip}%
\pgfsetrectcap%
\pgfsetroundjoin%
\pgfsetlinewidth{1.003750pt}%
\definecolor{currentstroke}{rgb}{0.000000,0.000000,0.000000}%
\pgfsetstrokecolor{currentstroke}%
\pgfsetdash{}{0pt}%
\pgfpathmoveto{\pgfqpoint{4.918009in}{2.521281in}}%
\pgfpathlineto{\pgfqpoint{6.040722in}{2.893770in}}%
\pgfusepath{stroke}%
\end{pgfscope}%
\begin{pgfscope}%
\pgfpathrectangle{\pgfqpoint{0.100000in}{0.100000in}}{\pgfqpoint{6.100000in}{3.688372in}}%
\pgfusepath{clip}%
\pgfsetrectcap%
\pgfsetroundjoin%
\pgfsetlinewidth{1.003750pt}%
\definecolor{currentstroke}{rgb}{0.000000,0.000000,0.000000}%
\pgfsetstrokecolor{currentstroke}%
\pgfsetdash{}{0pt}%
\pgfpathmoveto{\pgfqpoint{1.381991in}{1.325120in}}%
\pgfpathlineto{\pgfqpoint{1.943347in}{0.905415in}}%
\pgfusepath{stroke}%
\end{pgfscope}%
\begin{pgfscope}%
\pgfsetbuttcap%
\pgfsetroundjoin%
\definecolor{currentfill}{rgb}{0.996078,0.780392,0.721569}%
\pgfsetfillcolor{currentfill}%
\pgfsetlinewidth{1.003750pt}%
\definecolor{currentstroke}{rgb}{0.000000,0.000000,0.000000}%
\pgfsetstrokecolor{currentstroke}%
\pgfsetdash{}{0pt}%
\pgfsys@defobject{currentmarker}{\pgfqpoint{-0.098209in}{-0.098209in}}{\pgfqpoint{0.098209in}{0.098209in}}{%
\pgfpathmoveto{\pgfqpoint{0.000000in}{-0.098209in}}%
\pgfpathcurveto{\pgfqpoint{0.026045in}{-0.098209in}}{\pgfqpoint{0.051028in}{-0.087861in}}{\pgfqpoint{0.069444in}{-0.069444in}}%
\pgfpathcurveto{\pgfqpoint{0.087861in}{-0.051028in}}{\pgfqpoint{0.098209in}{-0.026045in}}{\pgfqpoint{0.098209in}{0.000000in}}%
\pgfpathcurveto{\pgfqpoint{0.098209in}{0.026045in}}{\pgfqpoint{0.087861in}{0.051028in}}{\pgfqpoint{0.069444in}{0.069444in}}%
\pgfpathcurveto{\pgfqpoint{0.051028in}{0.087861in}}{\pgfqpoint{0.026045in}{0.098209in}}{\pgfqpoint{0.000000in}{0.098209in}}%
\pgfpathcurveto{\pgfqpoint{-0.026045in}{0.098209in}}{\pgfqpoint{-0.051028in}{0.087861in}}{\pgfqpoint{-0.069444in}{0.069444in}}%
\pgfpathcurveto{\pgfqpoint{-0.087861in}{0.051028in}}{\pgfqpoint{-0.098209in}{0.026045in}}{\pgfqpoint{-0.098209in}{0.000000in}}%
\pgfpathcurveto{\pgfqpoint{-0.098209in}{-0.026045in}}{\pgfqpoint{-0.087861in}{-0.051028in}}{\pgfqpoint{-0.069444in}{-0.069444in}}%
\pgfpathcurveto{\pgfqpoint{-0.051028in}{-0.087861in}}{\pgfqpoint{-0.026045in}{-0.098209in}}{\pgfqpoint{0.000000in}{-0.098209in}}%
\pgfpathlineto{\pgfqpoint{0.000000in}{-0.098209in}}%
\pgfpathclose%
\pgfusepath{stroke,fill}%
}%
\begin{pgfscope}%
\pgfsys@transformshift{6.082692in}{2.668178in}%
\pgfsys@useobject{currentmarker}{}%
\end{pgfscope}%
\begin{pgfscope}%
\pgfsys@transformshift{3.396577in}{0.170930in}%
\pgfsys@useobject{currentmarker}{}%
\end{pgfscope}%
\begin{pgfscope}%
\pgfsys@transformshift{5.232789in}{3.659732in}%
\pgfsys@useobject{currentmarker}{}%
\end{pgfscope}%
\begin{pgfscope}%
\pgfsys@transformshift{5.531829in}{3.554806in}%
\pgfsys@useobject{currentmarker}{}%
\end{pgfscope}%
\begin{pgfscope}%
\pgfsys@transformshift{3.223448in}{1.650392in}%
\pgfsys@useobject{currentmarker}{}%
\end{pgfscope}%
\begin{pgfscope}%
\pgfsys@transformshift{4.477319in}{0.926400in}%
\pgfsys@useobject{currentmarker}{}%
\end{pgfscope}%
\begin{pgfscope}%
\pgfsys@transformshift{1.156399in}{1.193962in}%
\pgfsys@useobject{currentmarker}{}%
\end{pgfscope}%
\begin{pgfscope}%
\pgfsys@transformshift{2.887684in}{1.996649in}%
\pgfsys@useobject{currentmarker}{}%
\end{pgfscope}%
\begin{pgfscope}%
\pgfsys@transformshift{0.238293in}{2.678671in}%
\pgfsys@useobject{currentmarker}{}%
\end{pgfscope}%
\begin{pgfscope}%
\pgfsys@transformshift{0.542579in}{3.318722in}%
\pgfsys@useobject{currentmarker}{}%
\end{pgfscope}%
\begin{pgfscope}%
\pgfsys@transformshift{5.017690in}{1.377584in}%
\pgfsys@useobject{currentmarker}{}%
\end{pgfscope}%
\begin{pgfscope}%
\pgfsys@transformshift{2.284357in}{1.424800in}%
\pgfsys@useobject{currentmarker}{}%
\end{pgfscope}%
\begin{pgfscope}%
\pgfsys@transformshift{3.029335in}{0.181423in}%
\pgfsys@useobject{currentmarker}{}%
\end{pgfscope}%
\begin{pgfscope}%
\pgfsys@transformshift{1.140660in}{3.717442in}%
\pgfsys@useobject{currentmarker}{}%
\end{pgfscope}%
\begin{pgfscope}%
\pgfsys@transformshift{1.728248in}{1.545466in}%
\pgfsys@useobject{currentmarker}{}%
\end{pgfscope}%
\begin{pgfscope}%
\pgfsys@transformshift{0.395683in}{3.114115in}%
\pgfsys@useobject{currentmarker}{}%
\end{pgfscope}%
\begin{pgfscope}%
\pgfsys@transformshift{2.105983in}{1.204455in}%
\pgfsys@useobject{currentmarker}{}%
\end{pgfscope}%
\begin{pgfscope}%
\pgfsys@transformshift{0.930807in}{3.612515in}%
\pgfsys@useobject{currentmarker}{}%
\end{pgfscope}%
\begin{pgfscope}%
\pgfsys@transformshift{4.660940in}{0.679823in}%
\pgfsys@useobject{currentmarker}{}%
\end{pgfscope}%
\begin{pgfscope}%
\pgfsys@transformshift{0.280264in}{2.914755in}%
\pgfsys@useobject{currentmarker}{}%
\end{pgfscope}%
\begin{pgfscope}%
\pgfsys@transformshift{1.555119in}{2.468818in}%
\pgfsys@useobject{currentmarker}{}%
\end{pgfscope}%
\begin{pgfscope}%
\pgfsys@transformshift{1.969578in}{1.755319in}%
\pgfsys@useobject{currentmarker}{}%
\end{pgfscope}%
\begin{pgfscope}%
\pgfsys@transformshift{2.656846in}{0.317827in}%
\pgfsys@useobject{currentmarker}{}%
\end{pgfscope}%
\begin{pgfscope}%
\pgfsys@transformshift{3.207710in}{1.110021in}%
\pgfsys@useobject{currentmarker}{}%
\end{pgfscope}%
\begin{pgfscope}%
\pgfsys@transformshift{5.259020in}{1.220194in}%
\pgfsys@useobject{currentmarker}{}%
\end{pgfscope}%
\begin{pgfscope}%
\pgfsys@transformshift{3.538228in}{2.012388in}%
\pgfsys@useobject{currentmarker}{}%
\end{pgfscope}%
\begin{pgfscope}%
\pgfsys@transformshift{1.833174in}{0.674577in}%
\pgfsys@useobject{currentmarker}{}%
\end{pgfscope}%
\begin{pgfscope}%
\pgfsys@transformshift{0.217308in}{2.416355in}%
\pgfsys@useobject{currentmarker}{}%
\end{pgfscope}%
\begin{pgfscope}%
\pgfsys@transformshift{3.795297in}{0.270610in}%
\pgfsys@useobject{currentmarker}{}%
\end{pgfscope}%
\begin{pgfscope}%
\pgfsys@transformshift{4.566506in}{1.896969in}%
\pgfsys@useobject{currentmarker}{}%
\end{pgfscope}%
\begin{pgfscope}%
\pgfsys@transformshift{4.057613in}{1.482510in}%
\pgfsys@useobject{currentmarker}{}%
\end{pgfscope}%
\begin{pgfscope}%
\pgfsys@transformshift{4.272712in}{1.262165in}%
\pgfsys@useobject{currentmarker}{}%
\end{pgfscope}%
\begin{pgfscope}%
\pgfsys@transformshift{4.918009in}{2.521281in}%
\pgfsys@useobject{currentmarker}{}%
\end{pgfscope}%
\begin{pgfscope}%
\pgfsys@transformshift{5.794145in}{3.376431in}%
\pgfsys@useobject{currentmarker}{}%
\end{pgfscope}%
\begin{pgfscope}%
\pgfsys@transformshift{6.066953in}{2.426847in}%
\pgfsys@useobject{currentmarker}{}%
\end{pgfscope}%
\begin{pgfscope}%
\pgfsys@transformshift{0.726201in}{3.502343in}%
\pgfsys@useobject{currentmarker}{}%
\end{pgfscope}%
\begin{pgfscope}%
\pgfsys@transformshift{1.381991in}{1.325120in}%
\pgfsys@useobject{currentmarker}{}%
\end{pgfscope}%
\begin{pgfscope}%
\pgfsys@transformshift{1.943347in}{0.905415in}%
\pgfsys@useobject{currentmarker}{}%
\end{pgfscope}%
\begin{pgfscope}%
\pgfsys@transformshift{5.935795in}{3.156086in}%
\pgfsys@useobject{currentmarker}{}%
\end{pgfscope}%
\begin{pgfscope}%
\pgfsys@transformshift{4.771113in}{1.729087in}%
\pgfsys@useobject{currentmarker}{}%
\end{pgfscope}%
\begin{pgfscope}%
\pgfsys@transformshift{4.928502in}{3.675471in}%
\pgfsys@useobject{currentmarker}{}%
\end{pgfscope}%
\begin{pgfscope}%
\pgfsys@transformshift{6.040722in}{2.893770in}%
\pgfsys@useobject{currentmarker}{}%
\end{pgfscope}%
\end{pgfscope}%
\end{pgfpicture}%
\makeatother%
\endgroup%